\documentclass[]{article}
\usepackage[english]{babel}
\usepackage{amssymb}
\usepackage{amsmath,amsthm, dsfont}
\usepackage{subcaption}
\usepackage{bm}
\usepackage{tabularx,array}
\usepackage{tabularray}
\usepackage{graphicx, url}
\usepackage{enumitem} 

\usepackage{mathtools}
\usepackage{diagbox,hhline}
\usepackage{authblk}

\usepackage[textwidth=1.8cm, textsize=tiny]{todonotes}

\usepackage[backend=bibtex,url=false,doi=false,isbn=false]{biblatex}
\renewbibmacro*{urldate}{}
\theoremstyle{plain}
\newtheorem{theorem}{Theorem}[section]
\newtheorem{lemma}[theorem]{Lemma}

\newtheorem{proposition}[theorem]{Proposition}

\newtheorem{remark}[theorem]{Remark}

\newtheoremstyle{assumpstyle}
{\topsep}{\topsep}            
{\normalfont}      
{}              
{\bfseries}     
{.}             
{ }             
{\thmname{#1}\thmnumber{ #2}\thmnote{[#3]}}

\theoremstyle{assumpstyle}

\newtheorem{assumption}{}

\newtheorem{assumptionprimeP}{}[assumption]

\newtheorem{assumptionsecondP}{}[assumption]

\usepackage{color}

\DeclareMathOperator*{\argmin}{arg\,min}
\newcommand{\one}{\mathds{1}}

\newcommand{\E}{\mathbb{E}}

\newcommand{\id}[1]{\ensuremath{\mathds{1}_{#1}}}
\renewcommand{\P}{\mathbb{P}}

\newcommand{\ve}{\varepsilon}
\newcommand{\revar}{\color{black}}
\newcommand{\revarn}{\color{black}{}}
\newcommand{\revK}{\color{black}{}}

\DeclarePairedDelimiter{\abs}{\lvert}{\rvert}
\DeclarePairedDelimiter{\norm}{\lVert}{\rVert}

\title{Drift estimation for rough processes under small noise asymptotic : trajectory fitting method}

\author[1]{Arnaud Gloter}
\author[2,3]{Nakahiro Yoshida}
\affil[1]{Laboratoire de Math\'ematiques et Mod\'elisation d'Evry, Universit\'e d'Evry
	\footnote{
		Laboratoire de Math\'ematiques et Mod\'elisation d'Evry, CNRS, Univ Evry, 
		Universit\'e Paris-Saclay, 91037, Evry, France. e-mail: arnaud.gloter@univ-evry.fr}
}
\affil[2]{Graduate School of Mathematical Sciences, University of Tokyo
	\footnote{Graduate School of Mathematical Sciences, University of Tokyo: 3-8-1 Komaba, Meguro-ku, Tokyo 153-8914, Japan. e-mail: nakahiro@ms.u-tokyo.ac.jp}
}
\affil[3]{Japan Science and Technology Agency CREST
}

\begin{document}
\maketitle
\footnote{
	This work was in part supported by 
	Japan Science and Technology Agency CREST JPMJCR2115; 
	Japan Society for the Promotion of Science Grants-in-Aid for Scientific Research 
	No. 23H03354 (Scientific Research);  
	and by a Cooperative Research Program of the Institute of Statistical Mathematics. 
}
\begin{abstract}
We consider a process $X^\ve$ {\revarn that solves a} stochastic Volterra equation with an unknown parameter $\theta^\star$ in the drift function. The Volterra kernel is singular, {\revK and 
	includes as an example, $K_0(u)=c
u^{\alpha-1/2} \id{u>0}$ with $\alpha \in (0,1/2)$.}
It is assumed that the diffusion coefficient is proportional to $\ve \to 0$. From an observation of the path $(X^\ve_s)_{s\in[0,T]}$, we construct a Trajectory Fitting Estimator, which is shown to be consistent and asymptotically normal. We also specify identifiability conditions insuring the $L^p$ convergence of the estimator.
\end{abstract}

\section{Introduction}

Volterra equations have been the subject of many recent works as they found applications in several fields, as physics, mathematical finance, modeling in life science (e.g. see
\cite{bakerPerspectiveNumericalTreatment2000a} and references therein for the deterministic case or \cite{abijaberAffineVolterraProcesses2019,barndorff-nielsenModellingEnergySpot2013,eleuchCharacteristicFunctionRough2019}
in the stochastic situation).

The probabilistic properties of the stochastic Volterra models driven by Brownian motions are first studied in the seminal papers
\cite{bergerVolterraEquationsIto1980, bergerVolterraEquationsIto1980b}. The model is latter extended in numerous contexts, by considering more general {\revarn semimartingales} \cite{protterVolterraEquationsDriven1985},  singular kernels
\cite{cochranStochasticVolterraEquations1995}, non Lipschitz  coefficients \cite{promelStochasticVolterraEquations2023,wangExistenceUniquenessSolutions2008},
infinite dimensional settings
\cite{zhangStochasticVolterraEquations2010}, or {\revarn connections} with rough path theory \cite{promelParacontrolledDistributionApproach2021}.

The situation of singular kernels has attracted a lot of attention as it 
{\revarn leads} to so-called ``rough'' processes where the H\"older smoothness $\alpha$ is smaller than the Brownian regularity $\alpha=1/2$. Such models are shown to be useful for the purpose of modelling volatility processes.
The roughness condition $\alpha<1/2$ for the volatility process is empirically supported by several studies (\cite{fukasawaConsistentEstimationFractional2022, gatheralVolatilityRough2018}), and these
rough stochastic volatility models have been widely studied recently
(see \cite{gatheralQuadraticRoughHeston2020, eleuchRougheningHeston2018}).

Statistical questions are addressed in such models. In particular, the stochastic volatility models with rough {\revarn volatility} is considered in \cite{chongStatisticalInferenceRough2024a,chongStatisticalInferenceRough2024b}, where the Hurst index $\alpha$ is estimated from observations of the price process. {\revarn However, few results are available  for the estimation of the other parameters of the model. In stochastic volatility models, the statistical procedures typically proceed in two steps : first, the volatility process is reconstructed from the price process; second, the dynamic of the volatility is inferred based on this reconstruction.}

In this work, we consider a 
{\revarn  statistical} problem in a framework of
{\revarn a} rough process solution of Volterra equation with singular kernels 
{\revK $K$,}
\begin{equation*}
	X^\varepsilon_t=X^\varepsilon_0+\ve \int_0^t K(t-s)a(X^\ve _s) dB_s
	+ \int_0^t K(t-s) b(X_s^\ve,\theta^\star) ds.
\end{equation*}
{\revK A typical example of singular kernel is $K(u)=K_0(u):=\frac{1}{\Gamma(\alpha+1/2)}
u^{\alpha-1/2} \id{u>0}$.}
Our setting is such that the process is close to a deterministic Volterra model as we assume a small noise asymptotic $\ve \to 0$. 
The objective is to estimate a parameter $\theta$ in the drift function from a continuous observation of $X^\ve$ {\revarn over a fixed time} interval $[0,T]$. In the framework of semimartingale models, the small noise asymptotic has already been subject of many studies (see \cite{genon-catalotParametricInferenceSmall2021,	guyParametricInferenceDiscretely2014, hongweiParameterEstimationClass2010,kutoyantsIdentificationDynamicalSystems1994,
	 nakajimaMaximumLikelihoodType2025b,sorensenSmalldiffusionAsymptoticsDiscretely2003b,uchidaInformationCriteriaSmall2004}).
	{\revarn In particular, \cite{kutoyantsIdentificationDynamicalSystems1994} provides a  statistical study of drift estimation for bounded kernels with a constant diffusion coefficient.}

From a continuous observation of $X^\ve$, it is possible to consider the MLE estimator of the model. In the context of a Volterra Ornstein-Uhlenbeck model, this method is used in \cite{zunigaVolterraProcessesApplications2021} in order to estimate the drift parameter, when $T \to \infty$. The consistency of the estimator is proved in \cite{zunigaVolterraProcessesApplications2021}.
In practice, the MLE methods {\revarn require the} approximation of stochastic integrals, which can produce some difficulties {\revarn in applications to real data, when continuous observation is not available}.
For this reason, in this paper, we {\revarn focus on the trajectory fitting estimator as considered in \cite{kutoyantsStatisticalInferenceErgodic2004}, which is known to be less sensitive to discretization issues.} 
{\revarn The study of MLE methods and of the theoretical issues concerning their approximation form
high frequency data is addressed in \cite{gloterDriftEstimationRough2025a}.}

We prove that it is possible to estimate consistently, with convergence in $L^p$, the drift parameter $\theta$ under explicit sufficient identifiability conditions on the model. We also {\revarn establish} the asymptotic normality of the estimator with rate $\ve^{-1}$,
 and give some explicit expression for the {\revarn Gaussian} limiting variable as {\revarn a} function of $X^0$ and $\partial_\ve X^\ve \mid_{\ve=0}$ (see Section \ref{S: Central}).

The outline of the paper is the following. In Section
\ref{S: Model} we give the assumptions on the model and introduce the estimation procedure. We also discuss the validity of the identifiability condition.
{\revarn Consistency is established} in Section \ref{S: Consistency}, while the asymptotic normality is addressed in Section \ref{S: Central}.
{\revar In Section \ref{S : num sim}, we provide a Monte Carlo study to assess the quality of the estimator in practice.}
 The Section \ref{S: Appendix} contains a technical result on the dependence of the Volterra equation with respect to the parameter.


\section{Model and assumptions} \label{S: Model}
We let $(\Omega,(\mathcal{F}_t)_{t\ge 0}, \P)$ be a filtered probability space satisfying the usual conditions. We consider on this space an $r$-dimensional standard Brownian motion $B$. 

On this space, we define $(X^\ve_t)_{t\in [0,T]}$ as a $d$ dimensional process, solution on $[0,T]$ of the stochastic differential equation
\begin{equation}\label{Eq : Volterra SDE}
X^\varepsilon_t=X^\varepsilon_0+\ve \int_0^t K(t-s)a(X^\ve _s) dB_s
+ \int_0^t K(t-s) b(X_s^\ve,\theta^\star) ds.
\end{equation}
Here, $X_0^\ve =x_0 \in \mathbb{R}^d$, 
$a : \mathbb{R}^d \to \mathbb{R}^d \otimes\mathbb{R}^r$, and 
$b : \mathbb{R}^d \times \Theta \to \mathbb{R}^d$. Here, we use the tensor product $ \mathbb{R}^d \otimes\mathbb{R}^r$ to represent the set of matrices of size $d\times r$. 

The parameter $\theta^\star$ belongs to the compact set $\Theta \subset \mathbb{R}^{d_{\Theta}}$. 

{\revK 
The kernel $K : (0,\infty) \mapsto \mathbb{R}$ is of rough type. Let us denote $K_0(u)=\frac{1}{\Gamma(\alpha+1/2)}
u^{\alpha-1/2} \id{u>0}$ with $\alpha \in (0,1/2)$. We assume that $K$ is $\mathcal{C}^1$ and its behavior near $0$ can be compared to the one of $K_0$ :
\begin{equation} \label{Eq : comp K K0}
	\forall u\in(0,T],\quad \abs*{K(u)} \le c K_0(u), \quad  \abs*{K'(u)} \le c \abs*{K'_0(u)},
\end{equation}
for some constant $c>0$ depending on $T$.
}
Noteworthy, we have
\begin{equation} \label{Eq : Kernel L1 L2 bound}
	\sup_{t \in [0,T]} \int_0^t \abs*{K(t-s)}ds \le  c,\quad \sup_{t \in [0,T]} \int_0^tK(t-s)^2ds \le  c
\end{equation}
for some constant {\revarn $c$ depending on $T$}.

We assume $\ve\in [0,1]$ and consider the asymptotic framework $\varepsilon\to0$.

We introduce the following assumptions on the coefficients of the Volterra equation.
\begin{assumption} \label{Ass : global Lip} The functions $a$ and $b$ are continuous and
	there exists some $C>0$, such that for all $(x,x')\in \mathbb{R}^d$, $\theta \in \Theta$
	\begin{equation*}
		\abs{a(x)-a(x')} \le C \abs{x-x'},\quad
		\abs{b(x,\theta)-b(x',\theta)} \le C \abs{x-x'}.
	\end{equation*}
\end{assumption}
Let us stress that for $x\in \mathbb{R}^d$, we denote by $\abs{x}$ the Euclidean norm of $x$ given by $\abs{x}=\sqrt{x^* x}$, while for $m \in \mathbb{R}^{d}\otimes\mathbb{R}^r$ the quantity $\abs{m}$ is the operator norm of the $d\times r$ matrix $m$ {\revarn with respect to the Euclidian norm}.

From application of Theorem {\revarn 1.1} in \cite{wangExistenceUniquenessSolutions2008}, we know that under \ref{Ass : global Lip} {\revK and 
	\eqref{Eq : comp K K0},} the equation \eqref{Eq : Volterra SDE} admits a unique progressively measurable process $(X_t^\ve)_t$ as solution. 
Moreover, the solution admits finite moment{\revarn s} of any order by Lemma 2.2 of \cite{wangExistenceUniquenessSolutions2008},
\begin{equation} \label{Eq : moment Volterra}
	\forall p \ge1, \quad \sup_{t \in [0,T]} \E[\abs{X^\ve_t}^p] \le {c(p).}	
\end{equation}
The constant $c(p)$ depends on $p$, {\revarn $T$,} the coefficients of the SDE \eqref{Eq : Volterra SDE} and on the kernel $K$. As $\varepsilon \in [0,1]$, the constant 
$c(p)$ in \eqref{Eq : moment Volterra} is uniform with respect to $\varepsilon$.

{\revarn
	To get statistical results, we need more regularity than \ref{Ass : global Lip} on the coefficients of the model \eqref{Eq : Volterra SDE}. 
	With that regard, we introduce some notations.
For $E$ a finite dimensional space, and $k_1 \ge0, k_2\ge0$ integers, we denote by $\mathcal{C}^{k_1,k_2}_P(\mathbb{R}^d\times\overset{\circ}{\Theta},E)$ 
the set of functions $f : \mathbb{R}^d\times\overset{\circ}{\Theta} \to E$ such that
for all $0 \le i \le k_1$ and all $ 0 \le j \le k_2$, the partial derivatives $(x,\theta)\mapsto 
\frac{\partial^{i} f}{\partial x^i }$ 
and 
$(x,\theta)\mapsto \frac{\partial^{j} f}{\partial \theta^j }$
are well-defined and continuous on $\mathbb{R}^d \times \overset{\circ}{\Theta}$, and satisfy,
\begin{equation*}
	\sup_{\theta \in \overset{\circ}{\Theta}} 
	\norm*{
		\frac{\partial^{i} f(x,\theta)}{\partial x^i } }_E +
	\norm*{
		\frac{\partial^{j} f(x,\theta)}{\partial \theta^j } }_E \le c(1+|x|^c)
\end{equation*}
for some constant $c>0$. 
For $l\ge0$, we also denote $\mathcal{C}^{l}_P(\mathbb{R}^d\times\overset{\circ}{\Theta},E)$ the set of functions $f : \mathbb{R}^d\times\overset{\circ}{\Theta} \to E$ such that all partial derivatives
$(x,\theta)\mapsto 
\frac{\partial^{i+j} f}{\partial x^i \partial \theta^j }$ are well-defined for any $i+j \le l$ and continuous in 
$\mathbb{R}^d \times \overset{\circ}{\Theta}$, and satisfy,
$	\sup_{\theta \in \overset{\circ}{\Theta}} \norm{
	\frac{\partial^{i+j} f(x,\theta)}{\partial x^i \partial \theta^j } }_E \le c(1+|x|^c).$
Here  	$\frac{\partial^{i+j} f(x,\theta)}{\partial x^i \partial \theta^j } $
is any partial derivative $\frac{\partial^{i+j} f(x,\theta)}{\partial x_{u_1}\dots\partial x_{u_i} \partial \theta_{v_1}\dots\partial \theta_{v_j} } $ where
$(u_1,\dots,u_i) \in \{1,\dots,d\}^i$ and $(v_1,\dots,v_j) \in \{1,\dots,d_\Theta\}^j$.
For $f : \Theta \mapsto E$, we denote $\nabla_\theta f$ the matrix 
$\begin{bmatrix}
		\frac{\partial{f}}{\partial \theta_1}, \dots,\frac{\partial{f}}{\partial \theta_{d_\Theta}}	
	\end{bmatrix} 
	$ where we assimilate $\frac{\partial{f}}{\partial \theta_j}$ with its representation on some basis of $E$.

Lastly, we recall the definition of the H\"older norm. For a continuous function $f : [0,T]\to \mathbb{R}$ and $\beta\in[0,1)$, we define
$\norm{f}_{\mathcal{H}^\beta}=\sup_{s\in[0,T]}\abs*{f(s)}+\sup_{0\le s\neq t \le T}
\abs*{\frac{f(t)-f(s)}{(t-s)^\beta}}$. If $f \in \mathcal{C}^{k}([0,T],\mathbb{R})$, and $s =k+\beta$ with $k\in\mathbb{N}$ and $\beta\in[0,1)$, we
let $\norm{f}_{\mathcal{H}^s}=\sum_{l=0}^{k-1} \sup_{s\in[0,T]}  
\abs*{f^{(l)}(s)} + \norm{f^{(k)}}_{\mathcal{H}^\beta}$.
}


{\revK Using the two conditions in \eqref{Eq : comp K K0}, 	
	we can prove for $t\in[0,T]$ and $\delta\in(0,1)$, that $\int_0^t \abs{K(t+\delta-s)-K(t-s)} ds \le c \delta^{\alpha+1/2}$, $\int_0^t \abs{K(t+\delta-s)-K(t-s)}^2 ds \le c \delta^{2\alpha}$, $\int_{t}^{t+\delta} \abs{K(t+\delta-s)}ds \le c \delta^{\alpha+1/2}$, $\int_{t}^{t+\delta} \abs{K(t+\delta-s)}^2ds \le c \delta^{2\alpha}$ with some constant $c>0$. 
	We can apply Proposition 4.1 in 
\cite{richardDiscretetimeSimulationStochastic2021}, and we deduce} that for $q \ge 1$ 
\begin{equation} \label{Eq : moment increment Volterra}
	\E\left( \abs*{X^\ve_t-X^\ve_s}^q\right) \le  C \abs*{t-s}^{q \alpha}, {\revarn \quad \forall t,s \in [0,T],}
\end{equation}
where the constant {\revarn $C$ depends on $T$, the coefficient of the S.D.E., $x_0$ and $q$.} 

Using Kolmogorov's continuity criteria, it implies that a version of the process is such that the sample paths $t \mapsto X^\ve_t$ are a.s. $\alpha'$-H\"older for any $\alpha' \in (0,\alpha)$.

We assume that the set $\Theta$ is such that the following Sobolev embedding holds true. For $f \in \mathcal{C}^{1}(\overset{\circ}{\Theta})$, $p > d_\Theta$, we have
$$
\sup_{\theta \in \Theta} \abs{f(\theta)} \le c [\norm{f}_{{\revar \mathbf{L}}^p(\Theta)}+\norm{\nabla f}_{{\revar \mathbf{L}}^p(\overset{\circ}{\Theta})}].
$$
{\revarn This Sobolev embedding holds true provided that the compact parameter space
$\Theta$ satisfies the cone property, which is guaranteed, for example, when 
its boundary is 
 Lipschitz (see  Section 1.4.5 in \cite{mazyaSobolevSpacesApplications2011}).}

The true value of the parameter $\theta=\theta^\star$ is unknown. To estimate it, we introduce an estimator based on a trajectory fitting procedure (see  \cite{kutoyantsStatisticalInferenceErgodic2004}). Hence, we introduce the following deterministic trajectories : for $\theta \in \Theta$, we let $(X_t^0(\theta))_{t\in [0,T]}$ be the solution of 
\begin{equation} \label{Eq : Volterra no noise}
	X_t^0(\theta)=x_0 + \int_0^t K(t-s) b(X^0_s(\theta),\theta)ds.
\end{equation}
Under Assumption \ref{Ass : global Lip}, the solution is well defined, and from compactness of $\Theta$, and continuity of $b$, we have 
\begin{equation} \label{Eq : unif bound on X^0}
	\sup_{\theta\in\Theta} \sup_{t\in[0,T]} \abs{X^0_t(\theta)} \le {\revarn c},
\end{equation}
{\revar for $c$ some constant depending on $T$.}

The result \eqref{Eq : moment increment Volterra} also applies to the deterministic solution $X^0(\theta)$. It {\revar gives} that $\norm{X^{0}(\theta)}_{\mathcal{H}^{\alpha}} \le C$ 
 where the constant $C$ is independent of $\theta\in \Theta$ by compactness of $\Theta$.

From \eqref{Eq : Volterra SDE} and \eqref{Eq : Volterra no noise}, we have $X^0=X^0(\theta^\star)$.
We define for all $\theta\in \Theta$,
\begin{equation} \label{Eq : contrast TFE}
	Q_\ve(\theta)=\int_0^T \abs{X^\ve_t-X^0_t(\theta)}^2dt.
\end{equation}
We have $Q_0(\theta)=\int_0^T \abs{X^0_t-X^0_t(\theta)}^2dt=\int_0^T \abs{X^0_t(\theta^\star)-X^0_t(\theta)}^2dt$. In turn, $Q_0(\theta^\star)=0$.
We define the TFE as 
\begin{equation*}
	\hat{\theta}_\ve {\revarn \in } \argmin_{\theta \in \Theta} Q_\ve(\theta).
\end{equation*}
{\revarn We now introduce the following identifiability conditions, where $\rho,\rho'>0$.
\begin{assumption} \label{Ass : identif}
	If $\theta\neq\theta^\star$, then $\int_0^T \abs*{b(X^0_t(\theta),\theta)-b(X^0_t(\theta^\star),\theta^\star)}^2 dt >0$.
\end{assumption}
\begin{assumptionprimeP}[$\rho$] \label{Ass : identif_strong}
	There exists $c>0$ such that for all $\theta\in \Theta$
	\begin{equation*}
		\int_0^T \abs*{b(X^0_t(\theta),\theta)-b(X^0_t(\theta^\star),\theta^\star)}^2 dt \ge c \abs{\theta-\theta^\star}^\rho.
	\end{equation*}
\end{assumptionprimeP}
\begin{assumptionsecondP}[$\rho'$] \label{Ass : identif_strong_contrast}
	There exists $c'>0$ such that for all $\theta\in \Theta$,
	$Q_0(\theta) \ge c' \abs{\theta-\theta^\star}^{\rho^\prime}.$
\end{assumptionsecondP}
}
{\revarn Assumption \ref{Ass : identif} is minimal to allow the identification of the parameter $\theta$ from the model at $\ve=0$. 
	Assumption {\revarn \ref{Ass : identif_strong}[$\rho$]} provides a quantitative version of this property.} It is  	
	for instance satisfied with $\rho=2$ as soon as $(x,\theta)
\mapsto b(x,\theta)$ is a $\mathcal{C}^1$ function and the starting point of the process allows identifiability in the sense that $\theta \mapsto \abs*{\frac{b(x_0,\theta)-b(x_0,\theta^\star)}{\theta-\theta^\star}}$ is lower bounded by some {\revarn positive} constant. 
{\revarn Assumption \ref{Ass : identif_strong_contrast}[$\rho'$] is an identifiability condition of $\theta$ from the contrast function. }

We introduce some useful notations in the context of Volterra equations. We let $f \star g (t)=\int_0^t f(t-s) g(s)ds$ denote the convolution of two functions on $[0,\infty)$. It is known that the kernel $K_{\revK 0}(u)=\frac{1}{\Gamma(\alpha+1/2)}
u^{\alpha-1/2} \id{u>0}$ admits a first kind resolvent kernel $L_{\revK 0}$ with explicit expression 
$L_{\revK 0}(u)=\frac{u^{-\alpha-1/2}}{\Gamma(1/2-\alpha)} \id{u>0} $, which satisfies $L_{\revK 0}\star K_{\revK 0}(t)=K_{\revK 0}\star L_{\revK 0}(t)=1$ for all $t > 0$. {\revK We introduce the assumption that the kernel $K$ admits a first kind resolvent.
\begin{assumption}\label{Ass : kernel}
	There exists a measurable function $L : (0,\infty) \to \mathbb{R}$, which is in $\mathbf{L}^1((0,T])$ for all $T > 0$ and such that $L\star K (t)=1$ for all $t > 0$.
\end{assumption}
The existence of the first kind resolvent is discussed in \cite{gripenbergVolterraEquationsFirst1980}. By Theorem 2 in \cite{gripenbergVolterraEquationsFirst1980}, it is sufficient that $K$ is non-negative, non-increasing and satisfies $\lim_{u\to0+}K(u)=\infty$, in order for Assumption \ref{Ass : kernel} to hold.
}

\begin{lemma} \label{lem : mino Q}
	1)	Assume \ref{Ass : global Lip}, \ref{Ass : identif}, {\revK and \ref{Ass : kernel},} then $Q_0(\theta)=0$ implies $\theta=\theta^\star$.
	\\
{\revarn 2)  Assume \ref{Ass : global Lip}, \ref{Ass : identif_strong}[$\rho$] for some $\rho>0$, {\revK and \ref{Ass : kernel}}.
	  Then, \ref{Ass : identif_strong_contrast}[$\rho'$] holds true for any $\rho' \in [\frac{1+\alpha}{\alpha} \rho,\infty)$.}	
\end{lemma}
\begin{proof}
	We prove the first point. Assume that $0=Q_0(\theta)=\int_0^T \abs{X^0_t(\theta^\star)-X^0_t(\theta)}^2dt$. By continuity of the solutions of the Volterra equations, it yields  $X^0_t(\theta)=X^0_t(\theta^\star)$ for all $t\in [0,T]$. In turn by \eqref{Eq : Volterra no noise}, we deduce $(K(\cdot) \star b(X_\cdot^0(\theta),\theta)) (t) = (K(\cdot) \star b(X_\cdot^0(\theta^\star),\theta^\star)) (t)$ for all $t \in [0,T]$.  Set
	$\overline{b}(t)=b(X_t^0(\theta),\theta)-b(X_t^0(\theta^\star),\theta^\star)$. 
	It is 
	\begin{align} \label{Eq : K star b bar}
		\big(K \star \overline{b} \big)(t)&=X^0_t(\theta)- X^0_t(\theta^\star) ,
		\\ &= 0, \text{ for all $t \in [0,T]$.} \nonumber
	\end{align} 
	We apply a convolution with the function $L$, which gives $\big(L \star (K \star \overline{b} )\big)(t)  = 0$. 
	Using associativity of the convolution operator $(L\star K)\star \overline{b}=0$. 
	Since $L\star K=\one_{\{t>0\}}$, we have 
	$\int_0^t \overline{b}(t) dt =0$ for  $t\in[0,T]$. 
	We deduce $\overline{b}(t) = 0$ for $t \in[0,T]$, and in turn $0=\int_0^T {\revarn\abs*{\overline{b}(t)}^2} dt=\int_0^T\abs*{b(X^0_t(\theta),\theta)-b(X^0_t(\theta^\star),\theta^\star)}^2dt$. By \ref{Ass : identif}, we deduce $\theta=\theta^\star$.
	
	Now, we prove the second point of the lemma. Let us assume {\revarn \ref{Ass : identif_strong}[$\rho$]} with $\rho>0$. We have by \eqref{Eq : K star b bar},
	\begin{equation*}
		Q_0(\theta)=\int_0^T\abs{X_t^0(\theta)-X_t^0(\theta^\star)}^2 dt=\int_0^T \abs{K\star \overline{b}(t)}^2 dt.
	\end{equation*}
	We denote $\overline{B}(t) =\int_0^t \overline{b}(s) ds$. As $L\star K(t)=\one_{\{t>0\}}$, we have $\overline{B}(t)=\big((L\star K)\star b\big)(t)$. Hence,
	\begin{align*}
		\norm{\overline{B}}^2_{\mathbf{L}^2([0,T])}&= \norm{(L\star K) \star b}^2_{\mathbf{L}^2([0,T])}=\norm{L\star (K \star b)}^2_{\mathbf{L}^2([0,T])}
		\\
		& \le 	\norm{L}^2_{\mathbf{L}^1([0,T])} \times \norm{K\star\overline{b}}^2_{\mathbf{L}^2([0,T])} = \norm{L}^2_{\mathbf{L}^1([0,T])} \times Q_0(\theta) ,
	\end{align*}
	 where we used the Young's inequality. {\revK
	 	We deduce
	\begin{equation} \label{Eq : lemma mino Q Q B}
	 Q_0(\theta) \ge 
	 \norm{\overline{B}}^2_{\mathbf{L}^2([0,T])} 
	 \norm{L}_{\mathbf{L}^1([0,T])} ^{-2}. 
\end{equation}
}

Now, we compare the $\mathbf{L}^2$-norm of $\overline{B}$ with the 	$\mathbf{L}^2$-norm of $t \mapsto \overline{b}(t)=\frac{\partial}{\partial t}\overline{B}(t)$.
We use an interpolation inequality between Sobolev spaces $W^{s,p}$ as stated in \cite{mironescuGagliardoNirenbergInequalities2018},
\begin{align*}
&\norm{f}_{W^{1,p}} \le c \norm{f}^\gamma_{W^{s_1,p_1}} \times \norm{f}^{1-\gamma}_{W^{s_2,p_2}},
\\
& \text{ for } \left\{
\begin{aligned}
	\frac{1}{p}&=\frac{\gamma}{p_1}+\frac{1-\gamma}{p_2}, \quad p_1 \in [1,\infty) \cup \{+\infty\}, p_2 \in (1,\infty) \cup \{+\infty\}, \gamma \in (0,1),
	\\
	1&=\gamma s_1 + (1-\gamma) s_2, \quad 0\le s_1 < s_2.
\end{aligned}
\right.
\end{align*}
We use this inequality with $f=\overline{B}$ and $s_1=0$, $s_2=1+{\revarn \alpha}$, 
$p_1=2$ and $p_2=\infty$. In turn, we have  $\gamma={\revarn \frac{\alpha}{1+\alpha}}$ and $p=2\times{\revarn\frac{1+\alpha}{\alpha}}$. It entails,
\begin{equation*}
	\norm{\overline{B}}_{W^{1,p}} \le c \norm{\overline{B}}^\gamma_{W^{0,2}} \times \norm{\overline{B}}^{1-\gamma}_{W^{1+{\revarn \alpha},\infty}},
\end{equation*}
that we reinterpret as
\begin{equation*}
	\norm{\overline{B}}_{W^{1,p}} \le c \norm{\overline{B}}^\gamma_{\mathbf{L}^2([0,T])} \times \norm{\overline{B}}^{1-\gamma}_{\mathcal{H}^{1+{\revarn \alpha}}},
\end{equation*}
where $\norm{\overline{B}}_{\mathcal{H}^{1+{\revarn \alpha}}}$ is the $(1+{\revarn \alpha})$-H\"older norm of the function $\overline{B}$. Noting that $t \mapsto \overline{B}_t$ is 
$\mathcal{C}^1$ with $\overline{B}'=\overline{b}$, we have $	\norm{\overline{B}}_{W^{1,p}} 
{\revarn = 
\norm{\overline{B}}_{\mathbf{L}^p([0,T])}+
\norm{\overline{B}'}_{\mathbf{L}^p([0,T])
}} \ge 	\norm{\overline{b}}_{\mathbf{L}^p([0,T])}$.
We deduce that 
\begin{equation} \label{Eq : lemma mino Q b B}
	\norm{\overline{b}}_{\mathbf{L}^p([0,T])} \le c \norm{\overline{B}}^\gamma_{\mathbf{L}^2([0,T])} \times \norm{\overline{B}}^{1-\gamma}_{\mathcal{H}^{1+{\revarn \alpha}}}.
\end{equation}
We know that the ${\revarn \alpha}$-H\"older norm of $t\mapsto X^{0}_t(\theta)$ is upper bounded by a constant independent of $\theta$. From the assumption \ref{Ass : global Lip}, the function $t \mapsto \overline{b}(t)=b(X_t^0(\theta),\theta)-b(X_t^0(\theta^\star),\theta^\star)$ also admits a bounded ${\revarn \alpha}$-H\"older norm. In turn, $t \mapsto \overline{B}(t)=\int_0^t \overline{b}(s)ds$ admits a bounded  $(1+{\revarn \alpha})$-H\"older norm. We deduce from \eqref{Eq : lemma mino Q b B},
\begin{equation} \label{Eq : lemma mino Q b B deux}
	\norm{\overline{b}}_{\mathbf{L}^p([0,T])} \le c \norm{\overline{B}}^\gamma_{\mathbf{L}^2([0,T])}.
\end{equation}

Collecting \eqref{Eq : lemma mino Q Q B} and \eqref{Eq : lemma mino Q b B deux} gives $Q_0(\theta) \ge c \norm{\overline{b}}_{\mathbf{L}^p([0,T])}^{\frac{2}{\gamma}}$. Since $p=2\times\frac{1+{\revarn \alpha}}{{\revarn \alpha}}\ge 2$, we deduce
$Q_0(\theta) \ge c \norm{\overline{b}}_{\mathbf{L}^2([0,T])}^{\frac{2}{\gamma}}$. Now, Assumption {\revarn \ref{Ass : identif_strong}[$\rho$]} writes
$\norm{\overline{b}}_{\mathbf{L}^2([0,T])}^2 \ge c |\theta-\theta^\star|^\rho$. 
We deduce {\revarn $Q_0(\theta) \ge c  |\theta-\theta^\star|^{\frac{\rho}{\gamma}}
= c|\theta-\theta^\star|^{\frac{\rho(1+\alpha)}{\alpha}}
$.}
The lemma is proved.
\end{proof}
\section{Consistency}\label{S: Consistency}
{\revarn We establish the consistency of the estimator.}
\begin{proposition} \label{prop :  consistency}
	Assume \ref{Ass : global Lip}--{\revK \ref{Ass : kernel}} and that $b \in \mathcal{C}^{1}_P(\mathbb{R}^d\times\overset{\circ}{\Theta},\mathbb{R}^d)$,
	then \begin{equation*}
		\hat{\theta}_\ve \xrightarrow[\mathbb{P}]{\ve \to 0} \theta^\star.
	\end{equation*}
	
	{\revarn 	Furthermore, if we  assume {\revarn \ref{Ass : identif_strong_contrast}[$\rho'$]} for some $\rho'>0$, then, 
		for all $p \ge 1$, it holds that
		\begin{equation}
			\label{Eq : consistency Lp}
			\sup_{0<\ve\le 1}\E\left[\abs{\ve^{-1/\rho'}(\hat{\theta}_\ve-\theta^\star)}^p\right]<\infty.
	\end{equation}	}
	\begin{proof}
		{\revarn Let us write for $M>0$,} 
		\begin{align}\nonumber
			\P \Big(\abs{\hat{\theta}_\ve-\theta^\star}\ge M \Big)&\le \P \bigg( \inf_{\substack{|\theta-\theta^\star|\ge M \\ \theta \in \Theta}} Q_\ve(\theta) \le  Q_\ve(\theta^\star) \bigg)
			\\\nonumber
			&\le  
			\begin{multlined}[t]
				\P \bigg( \inf_{\substack{|\theta-\theta^\star|\ge M \\ \theta \in \Theta}} Q_0(\theta) - \sup_{\theta\in \Theta}  |Q_\ve(\theta)-Q_0(\theta)| \le 
				\\   Q_0(\theta^\star) +  \sup_{\theta{\revarn{\in\Theta}}}|Q_\ve(\theta)-Q_0(\theta)| \bigg)
			\end{multlined}
			\\ \label{eq: prop const majo proba}& \le \P \Big(2  \sup_{\theta\in\Theta}|Q_\ve(\theta)-Q_0(\theta)| \ge  
			\inf_{\substack{|\theta-\theta^\star|\ge M \\ \theta \in \Theta}}  Q_0(\theta) 
			\Big){\revar,}
		\end{align}
		{\revarn where in the last line we used $Q_0(\theta^\star)=0$.}
		
		Under \ref{Ass : identif}, we know by Lemma \ref{lem : mino Q} 1),
		that {\revarn for any fixed value $M>0$, we have,} 
		{\revarn $\inf_{|\theta-\theta^\star|\ge M, \theta \in \Theta} Q_0(\theta) >0$.} Thus, \eqref{eq: prop const majo proba} yields
		$		\P(\abs{\hat{\theta}_\ve-\theta^\star}\ge M) \le {\revarn\P} (\sup_{\theta\in\Theta}|Q_\ve(\theta)-Q_0(\theta)| \ge c(M))$ for some $c(M)>0$. By Proposition
		\ref{prop :  conv contrast} {\revarn below}, we know that this probability goes to zero. The consistency of $\hat{\theta}_\ve$ is shown.	
			
			{\revarn 	 Now, we prove the second point of the proposition. Using \ref{Ass : identif_strong_contrast}[$\rho'$], 
				we know that 
				{\revarn $\inf_{|\theta-\theta^\star|\ge M, \theta \in \Theta} Q_0(\theta) \ge c' \inf_{|\theta-\theta^\star|\ge M, \theta \in \Theta} |\theta-\theta^\star|^{\rho'}$.}
				Thus, if $\{\theta \in \Theta \mid \abs{\theta-\theta^\star}\ge M \} \neq \emptyset$, we have}
			\begin{equation*}
				\P(\abs{\hat{\theta}_\ve-\theta^\star}\ge M) \le  \P (2  \sup_{\theta\in\Theta}|Q_\ve(\theta)-Q_0(\theta)| \ge {\revarn c'}M^{{\rho'}} ).
			\end{equation*}
			Remark that if $\{\theta \in \Theta \mid \abs{ \theta-\theta^\star}\ge M \} = \emptyset$, the above inequality is true as the left hand side is zero, recalling $\hat{\theta}_\ve \in \Theta$. We deduce that
			\begin{equation*}
				\P(\abs{\hat{\theta}_\ve-\theta^\star}\ge M) \le 
				\frac{  \E \left[   \sup_{\theta\in\Theta}|Q_\ve(\theta)-Q_0(\theta)|^p \right]2^p}{{\revarn (c')^p}M^{p\rho'}} {\revarn,} 
			\end{equation*}
			for any $p \ge 1$.
			Using {\revarn \eqref{Eq : cv contrast unif} in Proposition \ref{prop :  conv contrast} below,} we get $	\P(\abs{\hat{\theta}_\ve-\theta^\star}\ge M) \le c(p)\frac{\ve^{p}}{M^{p\rho'}}$.
			Choosing $M=\ve^{1/\rho'}\times x$ for $x\ge 0$, {\revarn we deduce,} 
			\begin{equation*}
				\P(\ve^{-1/\rho'}\abs{\hat{\theta}_\ve-\theta^\star}\ge x) \le c(p) \frac{1}{x^{p\rho'}},
			\end{equation*}	
			where $p\ge 1$ can be fixed arbitrarily large.
			This is sufficient to get $\forall q \ge 1$,  $\sup_{0<\ve\le 1}\E\left[ \abs*{\ve^{-1/\rho'}(\hat{\theta}_\ve-\theta^\star)}^q  \right]\le c(q)$. The second point of the proposition is proved.
			%
			\end{proof}
		\end{proposition}
		{\revarn 
			\begin{remark}
				The rate appearing in \eqref{Eq : consistency Lp} depends on $\rho'$ as defined in \ref{Ass : identif_strong_contrast}[$\rho'$]. 
				Under \ref{Ass : identif_strong}[$\rho$], we can apply Lemma \ref{lem : mino Q} 
				to obtain that the convergence rate of the estimator is $\ve^{-r}$ for 
				$r=\frac{\alpha}{(1+\alpha)\rho}$. In particular, this rate deteriorates for small values of $\alpha$. Note that this rate is obtained without 
				any $\mathcal{C}^2$ assumption on the constrast function.
				In Theorem \ref{theo : CLT}, we show that, under additional assumptions, the convergence rate is $\ve^{-1}$.
			\end{remark}
			
		}
The next proposition shows the convergence of the contrast function $Q^\ve(\theta)$ uniformly in the $\mathbf{L}^p(\Omega)$ norm {\revarn and is a crucial step in the proof the convergence of the estimator. Its proof is postponed the end of the section, as it requires several lemmas}.
\begin{proposition} \label{prop :  conv contrast}
	Assume \ref{Ass : global Lip} and that $b \in \mathcal{C}^{1}_P(\mathbb{R}^d\times\overset{\circ}{\Theta},\mathbb{R}^d)$. 
	Then, for all $p \ge 1$, we have
	\begin{equation} \label{Eq : cv contrast unif}
		\E\left[\sup_{\theta\in\Theta}\abs*{Q_\ve(\theta)-Q_0(\theta)}^p\right] \le c(p) \ve^p,	
	\end{equation}	
	where the constant $c(p)$ is independent of $\ve$.	
\end{proposition}
{\revarn Before proving the uniform convergence of the contrast function, we consider
	as a preliminary, 
the convergence in $\mathbf{L}^p$ for a fixed value of $\theta$. This requires 
evaluating the distance between $X^\ve_t$ and $X^0_t(\theta^\star)$.}
\begin{lemma} \label{lem : maj Lp diff Xve X0}
	Assume \ref{Ass : global Lip}.
Let us set $\mathcal{E}^\ve_t=X^\ve_t-X^0_t(\theta^\star)$. Then, for any $p\ge2$, there exists $c(p)>0$ such that 
\begin{equation*}
	\forall t \in [0,T], \quad \E\left[\abs*{\mathcal{E}^\ve_t}^p\right]
	\le c(p) \ve^p, 
\end{equation*}
{\revarn where the constant $c(p)$ depends on $p$, $T$, and the coefficients $a$ and $b$.}
\end{lemma}
\begin{proof}
Using \eqref{Eq : Volterra SDE} and \eqref{Eq : Volterra no noise}, we have
\begin{multline*}
	\mathcal{E}^\ve_t=X^\ve_t-X^0_t(\theta^\star)=\ve \int_0^t K(t-s) a(X_s^\ve) dB_s
		 \\ +\int_0^t K(t-s) (b(X^\ve_s,\theta^\star) - b(X^0_s(\theta^\star),\theta^\star)  )ds.
\end{multline*}	
	We use the {\revarn Burkholder-Davis-Gundy} inequality to get,
\begin{multline*}
	\E\left[\abs{\mathcal{E}^\ve_t}^p\right] \le c(p) \ve^p \E\left[ \abs*{\int_0^t K(t-s)^2 |a(X_s^\ve)|^2 ds}^{p/2}\right]
	\\ +c(p)\E\left[ \abs*{ \int_0^t K(t-s) (b(X^\ve_s,\theta^\star) - b(X^0_s(\theta^\star),\theta^\star)  )ds}^p\right].
\end{multline*}	
We use Jensen's inequality with respect to the measures $\id{[0,T]}(s)K(t-s)^2 ds$ and
$\id{[0,T]}(s){\revK\abs*{K(t-s)}}ds$,
to deduce
\begin{align*}
		\E\left[\abs{\mathcal{E}^\ve_t}^p\right] &
		\begin{multlined}[t]
			\le c(p) \ve^p 
		\left(\int_0^T K(t-s)^2ds\right)^{p/2-1}
		\E\left[ \left(\int_0^t K(t-s)^2  \abs{a(X_s^\ve)}^p ds\right)\right]
		\\ +c(p)
		\left(\int_0^T {\revK\abs*{K(t-s)}}ds\right)^{p-1}\E\left[ \int_0^t 
		{\revK\abs*{K(t-s)}} \abs*{b(X^\ve_s,\theta^\star) - b(X^0_s(\theta^\star),\theta^\star)  }^pds \right]
		\end{multlined}
		\\
		&
	\begin{multlined}[t]
		\le c(p) \ve^p 
		\E\left[ \left(\int_0^t K(t-s)^2 \abs*{a(X_s^\ve)}^p ds\right)\right]
		\\+ c(p)
		\E\left[ \int_0^t{\revK\abs*{K(t-s)}} \abs*{b(X^\ve_s,\theta^\star) - b(X^0_s(\theta^\star),\theta^\star) }^pds \right]
	\end{multlined}
\end{align*}	
	where we used \eqref{Eq : Kernel L1 L2 bound} in the second inequality {\revarn and the constant $c(p)$ may vary from line to line}. 
	From \ref{Ass : global Lip}, we know that $a$ is sub-linear and use the Lipschitz property of $b$ to deduce,
\begin{multline*}
		\E\left[\abs{\mathcal{E}^\ve_t}^p\right] 
		\le c(p) \ve^p 
	\E\left[ \left(\int_0^t K(t-s)^2 (1+ \abs*{X_s^\ve}^p) ds\right)\right]
	\\+ c(p)
	\E\left[ \int_0^t {\revK\abs*{K(t-s)}} \abs*{X^\ve_s - X^0_s(\theta^\star) }^pds \right].	
\end{multline*}
From \eqref{Eq : moment Volterra} and \eqref{Eq : Kernel L1 L2 bound} again, we deduce
\begin{align*}
	\E\left[\abs{\mathcal{E}^\ve_t}^p\right] 
	&\le c(p) \ve^p + c(p)
	\E\left[ \int_0^t {\revK\abs*{K(t-s)}}\abs*{X^\ve_s - X^0_s(\theta^\star) }^pds \right]
	\\
	& \le c(p) \ve^p + c(p)
	\int_0^t {\revK\abs*{K(t-s)}} \E\left[\abs{\mathcal{E}^\ve_s}^p\right] ds.	
\end{align*}
Recalling that ${\revK\abs*{K(t-s)} \le c (t-s)^{\alpha-1/2}} $, we use the generalized Gronwall's inequality given by Corollary 2 in \cite{yeGeneralizedGronwallInequality2007}, and deduce
$	\E\left[\abs{\mathcal{E}^\ve_t}^p\right]  \le c(p)  \ve^p E_{\alpha+1/2}(c(p)\Gamma(\alpha+1/2)t^{\alpha+1/2})$ where 
{\revarn $E_{\alpha+1/2}$ is the Mittag-Leffler function defined by 
$E_{\alpha+1/2}(z)=\sum_{k=0}^\infty \frac{z^k}{\Gamma(k(\alpha+1/2)+1)}$. Since $t\le T$ and the Mittag-Leffler function is increasing, it yields the result.}
\end{proof}
{\revarn We can now state a pointwise convergence property of the contrast function.}
\begin{lemma}\label{lem : maj Lp Q}
	Assume \ref{Ass : global Lip}. We have for $p\ge 2$, 
	\begin{equation*}
		\sup_{\theta \in \Theta} \E\left[\abs*{Q_\ve(\theta)-Q_0(\theta)}^p\right]
		\le c(p) \ve^p
	\end{equation*}
\end{lemma}	
\begin{proof}
Using \eqref{Eq : Volterra SDE}, \eqref{Eq : Volterra no noise} and \eqref{Eq : contrast TFE}, we have
\begin{multline} \label{Eq : expression Q_ve proof}
	Q_\ve(\theta)=\int_0^T \Bigg| \int_0^t K(t-s)\left(b(X^\ve_s,\theta^\star) - b(X^0_s(\theta),\theta)\right)ds  
	\\ + \ve\int_0^t K(t-s)a(X^\ve_s) dB_s
	 \Bigg|^2 dt. 
\end{multline}	
	Expanding the square of the Euclidean norm, we get 	$Q_\ve(\theta)=\sum_{j=1}^3Q_\ve^{(j)}(\theta)$, where
	\begin{align*}
		Q_\ve^{(1)}(\theta)&=\int_0^T \bigg| \int_0^t K(t-s)\left(b(X^\ve_s,\theta^\star) - b(X^0_s(\theta),\theta)\right)ds  \bigg|^2 dt,
		\\
Q_\ve^{(2)}(\theta)&=	\ve^2	\int_0^T \bigg|\int_0^t K(t-s)a(X^\ve_s) dB_s
		\bigg|^2 dt,
		\\
Q_\ve^{(3)}(\theta)	&= 
\begin{multlined}[t]
2	\ve \int_0^T 
\bigg( \int_0^t K(t-s)\left(b(X^\ve_s,\theta^\star) - b(X^0_s(\theta),\theta)\right)ds\bigg)^*\times
\\\bigg(\int_0^t K(t-s)a(X^\ve_s) dB_s
\bigg) dt.
\end{multlined}
	\end{align*}
First, we focus on 	$Q_\ve^{(1)}(\theta)$. 
Let us denote $g_\ve(t)=\int_0^t K(t-s)\left(b(X^\ve_s,\theta^\star) - b(X^0_s(\theta),\theta)\right)ds $ and $g_0(t)=\int_0^t K(t-s)\left(b(X^0_s(\theta^\star),\theta^\star) - b(X^0_s(\theta),\theta)\right)ds $.
With this notation, $Q_\ve^{(1)}(\theta)=\norm{g_\ve}^2_{\mathbf{L}^2([0,T])}$ and $Q_0(\theta)=\norm{g_0}^2_{\mathbf{L}^2([0,T])}$.
 We deduce
\begin{align*}
	\abs*{Q_\ve^{(1)}(\theta)-Q_0(\theta)}&=\abs*{\norm*{g_\ve}^2_{\mathbf{L}^2([0,T])}-\norm*{g_0}^2_{\mathbf{L}^2([0,T])}}
	\\
	& = \abs*{\norm*{g_\ve}_{\mathbf{L}^2([0,T])}-\norm*{g_0}_{\mathbf{L}^2([0,T])}}\times\abs*{\norm*{g_\ve}_{\mathbf{L}^2([0,T])}+\norm*{g_0}_{\mathbf{L}^2([0,T])}}\\
	& \le \norm*{g_\ve-g_0}_{\mathbf{L}^2([0,T])}\times\abs*{\norm*{g_\ve}_{\mathbf{L}^2([0,T])}+\norm*{g_0}_{\mathbf{L}^2([0,T])}}.
\end{align*}
We apply the Young's inequality for convolution of functions, which gives for any $f \in \mathbf{L}^2([0,T])$, $\norm{K\star f}_{\mathbf{L}^2([0,T])} \le \norm{K}_{\mathbf{L}^1([0,T])} \norm{f}_{\mathbf{L}^2([0,T])}$. Using that 	$\norm{K}_{\mathbf{L}^1([0,T])} \le c$,
we deduce that 
\begin{align}\nonumber
	\norm*{g_\ve-g_0}_{\mathbf{L}^2([0,T])} &\le \norm*{K}_{\mathbf{L}^1([0,T])}  
	\norm*{ b(X^\ve_{\cdot},\theta^\star)- b(X^0_{\cdot}(\theta^\star),\theta^\star)}_{\mathbf{L}^2([0,T])} ,
	\\ \label{Eq : bound L2 norm g_ve g0}
	&\le  c 
	\norm*{ b(X^\ve_{\cdot},\theta^\star)- b(X^0_{\cdot}(\theta^\star),\theta^\star)}_{\mathbf{L}^2([0,T])} .
\end{align}
From similar control on $\norm{g_\ve}_{\mathbf{L}^2([0,T])}$ and $\norm{g_0}_{\mathbf{L}^2([0,T])}
$, we get,
\begin{multline*}
	\abs*{Q_\ve^{(1)}(\theta)-Q_0(\theta)} \le c 
	\norm*{ b(X^\ve_{\cdot},\theta^\star)- b(X^0_{\cdot}(\theta^\star),\theta^\star)}_{\mathbf{L}^2([0,T])}  \times
	\\
	\left[	\norm*{ b(X^\ve_{\cdot},\theta^\star)- b(X^0_{\cdot}(\theta),\theta)}_{\mathbf{L}^2([0,T])} +	\norm*{ b(X^0_{\cdot}(\theta^\star),\theta^\star)- b(X^0_{\cdot}(\theta),\theta)}_{\mathbf{L}^2([0,T])} \right].
\end{multline*}
By Cauchy-Schwarz's inequality,
\begin{multline*}
	\E\left[	\abs*{Q_\ve^{(1)}(\theta)-Q_0(\theta)}^p\right]\le 
	c 	\E\left[\norm*{ b(X^\ve_{\cdot},\theta^\star)- b(X^0_{\cdot}(\theta^\star),\theta^\star)}_{\mathbf{L}^2([0,T])}^{2p}\right]^{1/2}
	\times \\	\E\left[	
		\left(	\norm*{ b(X^\ve_{\cdot},\theta^\star)- b(X^0_{\cdot}(\theta),\theta)}_{\mathbf{L}^2([0,T])} +	\norm*{ b(X^0_{\cdot},\theta^\star)- b(X^0_{\cdot}(\theta),\theta)}_{\mathbf{L}^2([0,T])} \right)^{2p}
	\right]^{1/2}.
\end{multline*}
{\revarn Applying the Jensen} inequality $\norm{f}_{{\mathbf{L}^2([0,T])}}^{2p} \le T^{p-1}
 \norm{f}_{{\mathbf{L}^{2p}([0,T])}}^{2p}$ {\revarn together with} the sub-additivity of the square root function {\revarn gives,}
\begin{multline} \label{Eq : Q^1-Q_0}
	\E\left[	\abs*{Q_\ve^{(1)}(\theta)-Q_0(\theta)}^p\right]\le 
	c 	\E\left[\norm*{ b(X^\ve_{\cdot},\theta^\star)- b(X^0_{\cdot}(\theta^\star),\theta^\star)}_{\mathbf{L}^{2p}([0,T])}^{2p}\right]^{1/2}
	\times \\\left(	\E\left[	
		\norm*{ b(X^\ve_{\cdot},\theta^\star)- b(X^0_{\cdot}(\theta),\theta)}^{2p}_{\mathbf{L}^{2p}([0,T])}
	\right]^{1/2}
	+
	  	 \norm*{ b(X^0_{\cdot}(\theta^\star),\theta^\star)- b(X^0_{\cdot}(\theta),\theta)}_{\mathbf{L}^{2p}([0,T])}^{p}\right).
\end{multline}
Now, we use the Lipschitz property of $x\mapsto b(x,\theta^\star)$ to get,
\begin{align*}
	\E\left[\norm{ b(X^\ve_{\cdot},\theta^\star)- b(X^0_{\cdot}(\theta^\star),\theta^\star)}_{\mathbf{L}^{2p}([0,T])}^{2p}\right]
	&\le c 	\E \left[ \int_0^T \abs*{ X^\ve_{s}- X^0_s(\theta^\star)}^{2p} ds \right]
	\\ &=  c\int_0^T \E \left[ \abs*{ \mathcal{E}^\ve_s}^{2p} ds \right] \le
	c \varepsilon^{2p},
\end{align*}
where in the last line we use Lemma \ref{lem :  maj Lp diff Xve X0}.

From Assumption \ref{Ass : global Lip} and {\revarn the} compactness of $\Theta$, we have 
$\sup_{\theta} \abs{b(x,\theta)} \le c(1+\abs{x})$. We recall the existence of upper {\revarn bounds} for the $2p$-moments of $X^\ve$, independent of $\ve$, {\revarn as} given in \eqref{Eq : moment Volterra}. Together with $\sup_{\theta\in\Theta}\sup_{t \in [0,T]} |X^0_t| \le c$, this yields, $\sup_\theta \E\left[	
\norm*{ b(X^\ve_{\cdot},\theta^\star)- b(X^0_{\cdot}(\theta),\theta)}_{\mathbf{L}^{2p}([0,T])} \right] \le c$ and 
$  \norm*{ b(X^0_{\cdot}(\theta^\star),\theta^\star)- b(X^0_{\cdot}(\theta),\theta)}_{\mathbf{L}^{2p}([0,T])}^{2p} \le c$.
Plugging these upper bounds in \eqref{Eq : Q^1-Q_0}, we deduce,
\begin{equation}\label{Eq : controle Q^(1)}
		\E\left[	\abs*{Q_\ve^{(1)}(\theta)-Q_0(\theta)}^p\right]\le 
	c \varepsilon^{p},
\end{equation}
where the constant $c$ is independent on $\theta$.	

We now deal with $Q_\ve^{(2)}$. Using successively, Jensen and Burkholder-Davis-Gundy inequalities, we have
\begin{align*}
\E\left[\abs*{Q^{(2)}_\ve(\theta)}^p\right] &\le
 \ve^{2p} T^{p-1} 
\int_0^T \E \left[   \abs*{\int_0^tK(t-s)a(X^\ve_s)dB_s}^{2p}\right],
\\
& \le c(p) \ve^{2p} T^{p-1} \int_0^T \E \left[   \left(\int_0^tK(t-s)^2\abs{a(X^\ve_s)}^2 ds\right)^{p}\right].
\end{align*} 
Using  {\revar again the} Jensen inequality with the measure $K(t-s)^2\id{[0,t]}(s)ds$, we deduce
\begin{multline*}
	\E\left[\abs*{Q^{(2)}_\ve(\theta)}^p\right] \le c(p) \ve^{2p} T^{p-1}
	\int_0^T 
	\left(\int_0^t K(t-s)^2ds\right)^{p-1}
	\times \\ \E\left[ \int_0^t K(t-s)^2 \abs{a(X^\ve_s)}^{2p} ds \right]dt
\end{multline*}
From Assumption \ref{Ass : global Lip} and \eqref{Eq : moment Volterra}, we {\revar obtain}
\begin{equation} \label{Eq : controle Q^(2)}
	\E\left[\abs*{Q^{(2)}_\ve(\theta)}^p\right] \le c(p) \ve^{2p} T^{p} \norm*{K}_{\mathbf{L}^2([0,T])}^{2p} \le c(p) \ve^{2p}
\end{equation}

For the term $Q_\ve^{(3)}$, we remark that by Cauchy-Schwarz's inequality 
 \begin{align*}
 	\abs*{Q_\ve^{(3)}(\theta)}&\le 2 \left(Q_\ve^{(1)}(\theta)\right)^{1/2}\left(Q_\ve^{(2)}(\theta)\right)^{1/2}
 \\	&
 \le 2 \left(\abs*{Q_\ve^{(1)}(\theta)-Q_0(\theta)}\right)^{1/2}\left(Q_\ve^{(2)}(\theta)\right)^{1/2}
 +{\revarn 2}\left(Q_0(\theta)\right)^{1/2}\left(Q_\ve^{(2)}(\theta)\right)^{1/2},
\end{align*}
 where we used the 
 sub-additivity of the square root function in the second line.
 
 From \eqref{Eq : Volterra no noise} and \eqref{Eq : unif bound on X^0}, we have $\sup_{\theta\in\Theta} Q_0(\theta) \le c$. Then, using \eqref{Eq : controle Q^(1)} and \eqref{Eq : controle Q^(2)}, we deduce
 \begin{align}\nonumber
 	\E \left[\abs*{Q_\ve^{(3)}(\theta)}^{p}\right]	&
 	\begin{multlined}[t]
 	\le
 	c(p)   
 	\E \left[ \abs*{ Q_\ve^{(1)}(\theta)-Q_0(\theta)}^p\right]^{1/2}
 	 \times \E \left[ \left(Q_\ve^{(2)}(\theta)\right)^{p}\right]^{1/2}
 	\\ + c(p) \E \left[ \left(
 	 Q_\ve^{(2)}(\theta)
 	 \right)^{p/2} \right] 
 	\end{multlined}
 	\\ \label{Eq : controle Q^(3)}
 	&\le c(p)[ \ve^{p/2} \times \ve^p + \ve^p] \le c(p) \ve^p.
 \end{align}
 Collecting \eqref{Eq : controle Q^(1)}, \eqref{Eq : controle Q^(2)} and \eqref{Eq : controle Q^(3)}, we deduce the lemma.
\end{proof}
{\revarn In order to prove the uniform convergence of the contrast, it is necessary to obtain some control on its derivative with respect to the parameter $\theta$.}
If we assume that $b \in \mathcal{C}^{1}_P(\mathbb{R}^d\times\overset{\circ}{\Theta},\mathbb{R}^d)$, then the mapping $\theta \mapsto X^0_s(\theta)$ is $\mathcal{C}^1$ for all $s \in [0,T]$ and it is possible to define 
$Y^0_s(\theta)=\nabla_\theta X_s^0(\theta) \in \mathbb{R}^d \otimes \mathbb{R}^{d_\Theta}$. The process $(Y^0_s(\theta))_{s \in [0,T]}$ solves the Volterra equation obtained by formal derivation of \eqref{Eq : Volterra no noise},
\begin{equation}\label{Eq : ODE Volterra Y}
		Y^0_t(\theta)=\int_0^t K(t-s)\left[ (\nabla_xb)(X^0_s(\theta),\theta)Y^0_s(\theta)+(\nabla_\theta b)(X^0_s(\theta),\theta)\right]ds.
\end{equation}
Details on the justification of the smoothness of the deterministic Volterra equation with respect to the parameter $\theta$ can be found in Remark \ref{R : rem deriv} of the Appendix.

Moreover, recalling \eqref{Eq : unif bound on X^0} and using that 
$\sup_{\theta\in\Theta}(\abs{\nabla_\theta b(x,\theta)}+\abs{\nabla_x b(x,\theta)}) \le c(1+\abs{x}^c)$, we deduce that the mappings $s\mapsto (\nabla_xb)(X^0_s(\theta),\theta)$ and
$s\mapsto(\nabla_\theta b)(X^0_s(\theta),\theta)$ are bounded by constant independent of $\theta$. From this, we get 
\begin{equation}\label{Eq : bound derivative det flow} 
	\sup_{\theta \in \overset{\circ}{\Theta}} \sup_{t\in[0,T]}
\abs*{Y^0_t(\theta)} \le c.
\end{equation}
Hence, under the regularity assumption $b \in \mathcal{C}^{1}_P$, it is possible to differentiate the contrast function $Q_\ve(\theta)$
defined in \eqref{Eq : contrast TFE}. The following lemma gives a control on this derivative.
\begin{lemma} \label{lem : maj Lp nabla Q}
	Assume \ref{Ass : global Lip} and that $b \in \mathcal{C}^{1}_P(\mathbb{R}^d\times\overset{\circ}{\Theta},\mathbb{R}^d)$. 
	Then, for all $p \ge 1$, we have
	\begin{equation} \label{Eq : cv derivative contrast}
	\sup_{\theta\in\overset{\circ}{\Theta}}	\E\left[\abs*{\nabla_\theta[ Q_\ve(\theta)-Q_0(\theta)]}^p\right] \le c(p) \ve^p,	
	\end{equation}	
	where the constant $c(p)$ is independent of $\ve$.	
\end{lemma}
\begin{proof}
	We differentiate the expression of the contrast function given by \eqref{Eq : expression Q_ve proof}
	and use $\nabla_\theta X^0_s(\theta)=Y^0_s(\theta)$. It yields
	\begin{align}
&\nabla_\theta Q_\ve(\theta) \nonumber 
\label{Eq : Nabla Q_ve split}\\=
&\begin{multlined}[t]
-2
\int_0^T \Big( \int_0^t K(t-s)\left(b(X^\ve_s,\theta^\star) - b(X^0_s(\theta),\theta)\right)ds  \\+ \ve \int_0^t K(t-s)a(X^\ve_s) dB_s \Big)^*
\\ 
\times \Big( \int_0^t K(t-s)\big(\nabla_x b(X^0_s(\theta),\theta) Y^0_s(\theta) +
\nabla_\theta b(X^0_s(\theta),\theta) \big)ds \Big) dt,
	\end{multlined}	
\\   \nonumber
&\nabla_\theta Q_0(\theta)\\=&	
\label{Eq : Nabla Q_0 split}
\begin{multlined}[t]	
-2
\int_0^T \left( \int_0^t K(t-s)\left(b(X^0_s(\theta^\star),\theta^\star) - b(X^0_s(\theta),\theta)\right)ds \right)^* 
\\ \times 
\left( \int_0^t K(t-s)\left(\nabla_x b(X^0_s(\theta),\theta) Y^0_s(\theta) +
\nabla_\theta b(X^0_s(\theta),\theta) \right)ds \right) dt.
\end{multlined}
		\end{align}
We split $\nabla_\theta Q_\ve(\theta)=:\nabla_\theta Q_\ve(\theta)^{(1)}+\nabla_\theta Q_\ve(\theta)^{(2)}$, where 
\begin{align} \nonumber
&	\nabla_\theta Q_\ve(\theta)^{(1)}
\\	=& \label{Eq : Nabla Q^1}
	\begin{multlined}[t] -2
\int_0^T \left( \int_0^t K(t-s)\left(b(X^\ve_s,\theta^\star) - b(X^0_s(\theta),\theta)\right)ds  \right)^*
 \\\times \left( \int_0^t K(t-s)\left(\nabla_x b(X^0_s(\theta),\theta) Y^0_s(\theta) +
\nabla_\theta b(X^0_s(\theta),\theta) \right)ds \right) dt,
\end{multlined}
\\ \nonumber
& \nabla_\theta Q_\ve(\theta)^{(2)}
 \\=& \label{Eq : Nabla Q^2}
 \begin{multlined}[t]-2 \ve
 \int_0^T \left( \int_0^t K(t-s)a(X^\ve_s) dB_s \right)^*
  \\\times \left( \int_0^t K(t-s)\left(\nabla_x b(X^0_s(\theta),\theta) Y^0_s(\theta) +
 \nabla_\theta b(X^0_s(\theta),\theta) \right)ds \right) dt.
\end{multlined}
 \end{align}
Using \eqref{Eq : unif bound on X^0} with $b \in \mathcal{C}^{1}_P$, \eqref{Eq : bound derivative det flow}, and $\norm{K}_{\mathbf{L}^1([0,T])}<\infty$, we have $\abs{\nabla_\theta Q_\ve(\theta)^{(2)}}
\le c \ve \int_0^T \abs*{ \int_0^t K(t-s)a(X^\ve_s) dB_s }dt$. By computations similar to the study of the term $Q^{(2)}_\ve$ in the proof of Lemma \ref{lem : maj Lp Q}, we have
\begin{equation*}
	\sup_{\theta \in \Theta} \E \left[\abs{\nabla_\theta Q_\ve(\theta)^{(2)}}^p\right]
	\le c(p) \ve^p.
\end{equation*}
In order to prove \eqref{Eq : cv derivative contrast}, it remains to show
\begin{equation} \label{Eq : bound nabla Q^1 Q_0}
	\sup_{\theta \in \Theta} \E \left[\abs{\nabla_\theta Q_\ve(\theta)^{(1)}-\nabla_\theta Q_0(\theta)}^p\right]
	\le c(p) \ve^p.
\end{equation}
From \eqref{Eq : Nabla Q_0 split}--\eqref{Eq : Nabla Q^1}, we have
\begin{multline*}
	\nabla_\theta Q_\ve(\theta)^{(1)} - \nabla_\theta Q_0(\theta)= 
	-2 \int_0^T \left(g_\ve(t)-g_0(t)\right)^* \\ \times \left(\int_0^t K(t-s)\left(\nabla_x b(X^0_s(\theta),\theta) Y^0_s(\theta) +
	\nabla_\theta b(X^0_s(\theta),\theta) \right)ds\right) dt,
\end{multline*}
	where the functions $g_\ve$ and $g_0$ are defined in the proof of Lemma \ref{lem : maj Lp Q}.
We deduce that 
\begin{equation*}
	\abs*{	\nabla_\theta Q_\ve(\theta)^{(1)} - \nabla_\theta Q_0(\theta)}
	\le c \int_0^T \abs{g_\ve(t)-g_0(t)}dt
	\le c \sqrt{T} \norm*{ g_\ve-g_0 }_{\mathbf{L}^2([0,T])}.
\end{equation*}
Now, we use \eqref{Eq : bound L2 norm g_ve g0}, and the Lipschitz property of $b$ to deduce 
\begin{multline*}
\abs{	\nabla_\theta Q_\ve(\theta)^{(1)} - \nabla_\theta Q_0(\theta)} \le 
c \left(\int_0^T \abs{b(X^\ve_s(\theta^\star),\theta^\star)-b(X^0(\theta^\star),\theta^\star)}^2ds\right)^{1/2}
\\ \le
c \left(\int_0^T \abs{X^\ve_s(\theta^\star)-X^0(\theta^\star)}^2ds\right)^{1/2}.
\end{multline*}
From Jensen's inequality and Lemma \ref{lem : maj Lp diff Xve X0}, we get \eqref{Eq : bound nabla Q^1 Q_0} and the lemma follows.
\end{proof}
{\revarn We conclude the section by providing a proof of  Proposition \ref{prop :  conv contrast}.}
\begin{proof}[Proof of Proposition \ref{prop :  conv contrast}]
	We write the Sobolev embedding,
	\begin{equation*}
		\sup_{\theta\in\Theta}
		\abs*{Q_\ve(\theta)-Q_0(\theta)}^p \le c(p)
		\int_\Theta \abs*{Q_\ve(\theta)-Q_0(\theta)}^p d\theta +
		c(p) \int_{\overset{\circ}{\Theta}} \abs*{\nabla_\theta Q_\ve(\theta)-\nabla_\theta Q_0(\theta)}^p d\theta ,
	\end{equation*}
	valid as soon as $p>d_\Theta$. Now, Lemmas \ref{lem : maj Lp Q} and \ref{lem : maj Lp nabla Q} with the compactness of $\Theta$ give the result for $p$ large enough. Then, the {\revarn interpolation between} $\mathbf{L}^p(\Omega)$ norms gives the result \eqref{Eq : cv contrast unif} for all $p \ge 1$.
\end{proof}
\section{Asymptotic normality of the estimator} \label{S: Central}

 The first step is {\revarn to} expand asymptotically the process $X^\ve$ around $\ve=0$. 
Assume \ref{Ass : global Lip} and  $b \in \mathcal{C}^{1,0}_P(\mathbb{R}^d\times\overset{\circ}{\Theta},\mathbb{R}^d)$ then, we can define $Z^0$ as the $\mathbb{R}^d$ valued process, solution of the linear Volterra equation
\begin{equation} \label{Eq : SDE partial X epsilon}
	Z^0_t=\int_0^t K(t-s) \nabla_x b(X^0_s(\theta^\star),\theta^\star) Z^0_s ds + \int_0^t K(t-s)a(X^0_s(\theta^\star))dB_s.
\end{equation}
\begin{lemma} \label{lem : partial ve X zero}
	For $\ve>0$, set $\overline{E}^\ve_t=X^\ve_t-X^0_t-\ve Z^0_t$. Assume that $b \in \mathcal{C}^{2,0}_P(\mathbb{R}^d\times\overset{\circ}{\Theta},\mathbb{R}^d)$, then for all $p \ge1$,
	\begin{equation*}
		\sup_{t\in[0,T]}\E\left[\abs[\big]{\overline{E}^\ve_t}^p\right] \le c(p)\ve^{2p}
			\end{equation*}
\end{lemma}
\begin{proof}
	From \eqref{Eq : Volterra SDE} and \eqref{Eq : SDE partial X epsilon}, we have
	\begin{align*}
	\overline{E}_t^\ve&=X^\ve_t-X^0_t-\ve Z^0_t
	\\
		&=\begin{multlined}[t]\int_0^t 
	K(t-s)\big[b(X^\ve_s,\theta^\star)-b(X^0_s(\theta^\star),\theta^\star)-\nabla_x b(X^0_s(\theta^\star),\theta^\star)\ve Z^0_s\big] ds\\  + 
	\ve \int_0^t K(t-s)\big[a(X^\ve_s)-a(X^0_s(\theta^\star))\big] dB_s
\end{multlined}
\\
&=
\begin{multlined}[t]
	\int_0^t 
	K(t-s)\Bigl[\nabla_x b(X^0_s(\theta^\star),\theta^\star)\overline{E}_s^\ve \\ + 
	\Bigl(\int_0^1 (\nabla_x b(\tilde{X}^\xi_s,\theta^\star)-\nabla_x b(X^0_s(\theta^\star),\theta^\star))d\xi\Bigr)(X^\ve_s-X^0_s(\theta^\star) )\Bigr] ds\\  + 
	\ve \int_0^t K(t-s)\big[a(X^\ve_s)-a(X^0_s(\theta^\star))\big] dB_s
\end{multlined}
	\end{align*}
	where $\tilde{X}^\xi_s:=(1-\xi)X^0_s(\theta^\star)+\xi X^\ve_s$.	
Application of the Burkholder-Davis-Gundy inequality yields
\begin{multline*}
	\E \left(	\abs[\big]{\overline{E}_t^\ve}^p \right) \le c(p)
	\E \biggl( \int_0^t 
{\revK\abs*{K(t-s)}} \Bigl\lvert \nabla_x b(X^0_s(\theta^\star),\theta^\star)\overline{E}_s^\ve 
	\\+
	\Bigl(\int_0^1
	(\nabla_x b(\tilde{X}^\xi_s(\theta^\star),\theta^\star)-\nabla_x b(X^0_s,\theta^\star) ) d\xi \Bigr)(X^\ve_s-X^0_s(\theta^\star) )\Bigr\rvert ds \biggr)^p  \\+ 
c(p)	\ve^p \E\left(  \int_0^t K(t-s)^2 \big[a(X^\ve_s)-a(X^0_s(\theta^\star))\big]^2 ds \right)^{p/2}.
\end{multline*}	
Using $K\in \mathbf{L}^2([0,T])$, Jensen's inequality and Fubini's theorem, we get,
$	\E \left(	\abs{\overline{E}_t^\ve}^p \right) \le c(p)\sum_{l=1}^3F^{\ve,l}$, with
\begin{align*}
F^{\ve,1}& =	
	\left( \int_0^T {\revK\abs*{K(s)}}ds  \right)^{p-1}   \int_0^t 
	{\revK \abs*{K(t-s)}} \E \left(\abs{\nabla_x b(X^0_s(\theta^\star),\theta^\star)	\overline{E}_t^\ve}^p \right) ds, \\ 
F^{\ve,2}& =
\begin{multlined}[t]
	\left( \int_0^T {\revK\abs*{K(s)}}ds \right)^{p-1}   \int_0^t {\revK \abs*{K(t-s)}} \\ \times
	\int_0^1 \E \big[ \abs{(\nabla_x b(\tilde{X}^\xi_s,\theta^\star)-\nabla_x b(X^0_s(\theta^\star),\theta^\star))(X^\ve_s-X^0_s(\theta^\star) )}^p\big]d\xi ds  	,
\end{multlined}
	\\
F^{\ve,3}& =
\ve^p  \left(  \int_0^T K(s)^2 ds \right)^{p/2 -1}   \int_0^t K(t-s)^2 \E \left( \abs*{a(X^\ve_s)-a(X^0_s(\theta^\star))}^p  \right)ds.
\end{align*}	
 As $\nabla_x b(\cdot,\theta^\star)$ is at most {\revarn of} polynomial growth, we get that  $\sup_{t \in [0,T]} 
\abs*{\nabla_x b(X_s^0(\theta^\star),\theta^\star)} < \infty$,  using \eqref{Eq : unif bound on X^0}. It gives
$F^{\ve,1} \le c(p) \int_0^1 {\revK\abs*{K(t-s)}}	\abs{\overline{E}_s^\ve}^p  ds$.
 Using that $b \in \mathcal{C}^{2,0}$ and {\revarn that  $\tilde{X}^\xi_s$
 belongs to the segment joining the points $X^0_s(\theta^\star)$ and $X^\ve_s$,} 
  we can write
$\E \big[ \abs{(\nabla_x b(\tilde{X}^\xi_s,\theta^\star)-\nabla_x b(X^0_s(\theta^\star),\theta^\star))(X^\ve_s-X^0_s )}^p\big]  \le 
c \E \big[ \abs{(1+\abs{X^0_s}^c+\abs{X^\ve_s}^c) \abs{X^\ve_s-X^0_s }^2}^p\big] \le  
c \E \big[ (1+\abs{X^0_s(\theta^\star)}^c+\abs{X^\ve_s}^c)^{2p} \big]^{1/2} \times \E \big[ \abs{X^\ve_s-X^0_s(\theta^\star)}^{4p}\big]^{1/2} \le c \ve^{2p}
$, where we used \eqref{Eq : moment Volterra}, \eqref{Eq : unif bound on X^0} and Lemma \ref{lem : maj Lp diff Xve X0}. It shows $F^{\ve,2} \le c(p) \ve^{2p}$.
From Assumption \ref{Ass : global Lip} and Lemma \ref{lem : maj Lp diff Xve X0}, we have $ \E \left( \abs*{a(X^\ve_s)-a(X^0_s(\theta^\star))}^p \right) \le c\ve^p$. Then, we deduce $F^{3,\ve} \le c(p) \ve^{2p}$. Collecting these results implies,
 \begin{equation*}
 	 	\E \left(	\abs[\big]{\overline{E}_t^\ve}^p \right) \le c(p)
 	  \int_0^t 
 	 {\revK\abs*{K(t-s)}} 	\E \left( \abs[\big]{\overline{E}_s^\ve}^p \right) ds + c(p)    \ve^{2p} .
 	 \end{equation*}	
 	where we used $K \in \mathbf{L}^2([0,T])$. Eventually, {\revK we  use \eqref{Eq : comp K K0} and} resort to the Gronwall inequality given by Corollary 2 in \cite{yeGeneralizedGronwallInequality2007}, and deduce
 	$\sup_{t \in [0,T ]} 	\E \left(	\abs[\big]{\overline{E}_t^\ve}^p \right) \le c(p) \ve^{2p}$.		
\end{proof}	 

\begin{lemma}\label{lem :  cv grad Q}
Assume \ref{Ass : global Lip} and $b \in \mathcal{C}^{2,0}_P(\mathbb{R}^d\times\overset{\circ}{\Theta},\mathbb{R}^d) \cap\mathcal{C}^{1}_P(\mathbb{R}^d\times\overset{\circ}{\Theta},\mathbb{R}^d) $. Then, for any $p \ge 1$,
\begin{equation*}
	\frac{1}{\ve} \nabla_\theta Q_\ve(\theta^\star)
	\xrightarrow[L^p]{\ve\to 0} \dot{Q}_0^{1}+\dot{Q}_0^{2},
\end{equation*}
	where
\begin{align*}
\dot{Q}_0^{1}&:=-2
\begin{multlined}[t] \int_0^T \left(\int_0^t K(t-s) \nabla_x b(X^0_s(\theta^\star),\theta^\star) Z^0_s ds\right)^*
\\\times 
\left( \int_0^t K(t-s)\left(\nabla_x b(X^0_s(\theta^\star),\theta^\star) Y^0_s(\theta^\star) +
\nabla_\theta b(X^0_s(\theta^\star),\theta^\star) \right)ds \right) dt ,
\end{multlined}
\\
\dot{Q}_0^{2}&:= -2
\begin{multlined}[t]\int_0^T \left(\int_0^t K(t-s) a(X^0_s) dB_s\right)^* 
\\\times 
\left( \int_0^t K(t-s)\left(\nabla_x b(X^0_s(\theta^\star),\theta^\star) Y^0_s(\theta^\star) +
\nabla_\theta b(X^0_s(\theta^\star),\theta^\star) \right)ds \right)  dt.
\end{multlined}
\end{align*}
\end{lemma}
\begin{proof}
{\revarn Recalling the notations \eqref{Eq : Nabla Q^1}--\eqref{Eq : Nabla Q^2} introduced in the proof of Lemma \ref{lem : maj Lp nabla Q}, we write
\begin{equation*}
	\frac{1}{\ve}\nabla_\theta Q_\ve(\theta^\star)= \frac{1}{\ve}
	\nabla_\theta Q_\ve(\theta^\star)^{(1)}+\frac{1}{\ve}\nabla_\theta Q_\ve(\theta^\star)^{(2)}
\end{equation*}
with
\begin{align*}
\ve^{-1}
\nabla_\theta Q_\ve(\theta^\star)^{(1)} &=\begin{multlined}[t]
	-2
	\int_0^T \left(\int_0^t K(t-s)\frac{1}{\ve}\left(b(X^\ve_s,\theta^\star) - b(X^0_s(\theta^\star),\theta^\star)\right)ds\right)^* 
	\\ \times \left( \int_0^t K(t-s)\left(\nabla_x b(X^0_s(\theta^\star),\theta^\star) Y^0_s(\theta^\star) +
	\nabla_\theta b(X^0_s(\theta^\star),\theta^\star) \right)ds \right) dt,
\end{multlined}
\\
\ve^{-1}
\nabla_\theta Q_\ve(\theta^\star)^{(2)}
&=\begin{multlined}[t]
	-2
	\int_0^T \left( \int_0^t K(t-s)a(X^\ve_s) dB_s \right)^*
	\\ \times \left( \int_0^t K(t-s)\left(\nabla_x b(X^0_s(\theta^\star),\theta^\star) Y^0_s(\theta^\star) +
	\nabla_\theta b(X^0_s(\theta^\star),\theta^\star) \right)ds \right) dt.
\end{multlined}
\end{align*}
}
We start with the convergence 
${\revarn \ve^{-1}
	\nabla_\theta Q_\ve(\theta^\star)^{(2)}}\to \dot{Q}^2_0$
in $\mathbf{L}^p$, and {w.l.o.g.} assume that $p\ge2$.  As $b \in \mathcal{C}^{1}_P$ and the paths $s \mapsto X_s^0(\theta)$ $s \mapsto Y_s^0(\theta)$ are bounded, we have that $\eta_t:=
 \int_0^t K(t-s)\left(\nabla_x b(X^0_s(\theta^\star),\theta^\star) Y^0_s(\theta^\star) +
\nabla_\theta b(X^0_s(\theta^\star),\theta^\star)\right)ds$ is bounded by a constant independent of {\revarn $t\in[0,T]$.}
Thus, we deduce
\begin{align*}
\abs*{{\revarn \ve^{-1}	\nabla_\theta Q_\ve(\theta^\star)^{(2)} }-\dot{Q}^{2}_0}^p &=2^p
	 \abs*{\int_0^T \big(\int_0^t K(t-s)[a(X^\ve_s)-a(X^0_s(\theta^\star))]dB_s\big)^* \eta_t dt}^p
	\\
	& \le C T^{p-1} \int_0^T \abs[\big]{\int_0^t K(t-s)[a(X^\ve_s)-a(X^0_s(\theta^\star))]dB_s}^p dt,
\end{align*}
where we used Jensen's inequality. We use Burkholder-Davis-Gundy inequality and get,
\begin{align*}
	\E\left[\abs*{
		{\revarn \ve^{-1}\nabla_\theta Q_\ve(\theta^\star)^{(2)} }
		-\dot{Q}^{2}_0}^p \right] 
	& \le C T^{p-1} \int_0^T \E \left[\abs[\big]{\int_0^t K(t-s)^2[a(X^\ve_s)-a(X^0_s(\theta^\star))]^2ds}^{p/2}\right] dt
	\\
	& \le \begin{multlined}[t]
		C T^{p-1} \int_0^T \E \Bigl[\bigl(\int_0^t K(t-s)^2 ds \bigr)^{p/2-1} \\ \times \bigl( \int_0^t K(t-s)^2 [a(X^\ve_s)-a(X^0_s(\theta^\star))]^pds\bigr)\Bigr] dt,
	\end{multlined}
\end{align*}
where we used again Jensen's inequality in the last line.  From $K \in \mathbf{L}^2([0,T])$, we have
\begin{equation*}
	\E\left[	\abs*{{\revarn \ve^{-1}\nabla_\theta Q_\ve(\theta^\star)^{(2)} }
		-\dot{Q}^{2}_0}^p \right] \le 
	C \int_0^T \left(\int_0^t K(t-s)^2 \E\big[\abs{a(X^\ve_s)-a(X^0_s(\theta^\star))}^p\big]ds\right)dt.
\end{equation*}
Recalling Assumption \ref{Ass : global Lip} and Lemma \ref{lem : maj Lp diff Xve X0}, we deduce $	
\E\left[	\abs[\big]{
	{\revarn \ve^{-1}\nabla_\theta Q_\ve(\theta^\star)^{(2)} }
	-\dot{Q}^{2}_0}^p \right] \le c \ve^{p}$.

We now focus on the convergence 
${\revarn \ve^{-1}\nabla_\theta Q_\ve(\theta^\star)^{(1)} }\to \dot{Q}^1_0$
in $\mathbf{L}^p$. We write
\begin{equation} \label{Eq : diff Q1 ve Q1 zero}
	\abs*{
		{\revarn \ve^{-1}\nabla_\theta Q_\ve(\theta^\star)^{(1)} }
		- \dot{Q}^1_0}^p = 2^p\abs*{
	\int_0^T \big(\int_0^t K(t-s) \frac{G_s}{\ve}ds\big)^*\eta_tdt}^p
\end{equation}
where $G_s:=b(X^\ve_s,\theta^\star)-b(X^0_s(\theta^\star),\theta^\star)-\ve \nabla_x b(X^0_s(\theta^\star),\theta^\star)Z^0_s$.
From Jensen's inequality and $\norm{\eta}_{\infty} \le C$, we deduce that the right-hand side of \eqref{Eq : diff Q1 ve Q1 zero} is no greater than
\begin{multline*}
	C \int_0^T \abs*{\int_0^t K(t-s)\frac{G_s}{\ve}ds}^pdt \\\le C\int_0^T \big(\int_0^t {\revK \abs*{K(t-s)}} ds\big)^{p-1}\big(\int_0^t {\revK \abs*{K(t-s)}}
	\abs[\big]{\frac{G_s}{\ve}}^pds\big) dt.
\end{multline*}
We deduce that 
\begin{equation}\label{Eq : Lp norm Q1 ve Q1 zero}
	\E\left[ 	\abs*{ {\revarn \ve^{-1}\nabla_\theta Q_\ve(\theta^\star)^{(1)} }
		- \dot{Q}^1_0}^p \right]
		\le \frac{C}{\ve^p} \int_0^T \big(\int_0^t {\revK \abs*{K(t-s)}} \E\left[\abs{G_s}^p\right] ds \big) dt.
\end{equation}
Using {\revarn the} order two Taylor's formula {\revarn for the $\mathbb{R^d}$ valued 
function $x\mapsto b(x,\theta^\star)$,} we can write
{\revarn 
\begin{multline*}
	b(X^\ve_s,\theta^\star)-b(X^0_s(\theta^\star),\theta^\star)=\nabla b_x(X^0_s(\theta^\star),\theta^\star)(X_s^\ve-X^0_s(\theta^\star))
	\\ +
\int_0^1(1-\xi) D^2_{xx} b(\widetilde{X}^\xi_s,\theta^\star)\big[X_s^\ve-X^0_s(\theta^\star),X_s^\ve-X^0_s(\theta^\star)\big]d\xi,
\end{multline*}
where we use the notation
$D^2_{xx} b(x ,\theta^\star)\big[u,v]:=
\begin{bmatrix} \sum_{1\le j,k\le d} \frac{\partial^2}{\partial x_j\partial x_k} b_i(x,\theta^\star) u_j{\revarn v_k}
\end{bmatrix}_{1\le i\le d} \in \mathbb{R}^d$, and with $\widetilde{X}^\xi_s=\xi X_s^\ve+(1-\xi)X^0_s(\theta^\star)$.}
Recalling the definition
of $\overline{E}^\ve_s$ in the statement of Lemma \ref{lem : partial ve X zero}, we deduce
\begin{align*}
	\abs*{G_s}&=\abs*{\nabla_x b(X^0_s(\theta^\star),\theta^\star)\overline{E}_s^\ve + 
		\int_0^1 (1-\xi) D^2_{xx} b(\widetilde{X}^\xi_s,\theta^\star)\big[X_s^\ve-X^0_s(\theta^\star),X_s^\ve-X^0_s(\theta^\star)\big] d\xi		
		}
	\\
	&\le C \abs{\overline{E}_s^\ve } + C (1+\abs{X_s^0(\theta^\star)}^C+\abs{X_s^\ve}^C)
\abs{X_s^\ve-X^0_s(\theta^\star)}^2,	
\end{align*}
where we used that {\revarn $\widetilde{X}^\xi_s$ lies in the segment $[X^0(\theta^\star),X_s^\ve]$}, together with $b \in \mathcal{C}^{2,0}_P$ and \eqref{Eq : unif bound on X^0}.
Now, Cauchy-Schwarz's inequality gives,
\begin{multline*}
	\E\left[\abs{G_s}^p\right]\le C \E\left[\abs{\overline{E}_s^\ve}^p \right]
	\\ +
	\E \left[   (1+\abs{X_s^0(\theta^\star)}^C+\abs{X_s^\ve}^C)^{2p} \right]^{1/2}
 		\E \left[ \abs{X_s^\ve-X^0_s(\theta^\star)}^{4p}\right]^{1/2}.
\end{multline*}
Using \eqref{Eq : moment Volterra}, with Lemmas \ref{lem : maj Lp diff Xve X0} and \ref{lem : partial ve X zero} , we deduce
$		\E\left[\abs{G_s}^p\right] \le C \ve^{2p}$ .
Then, \eqref{Eq : Lp norm Q1 ve Q1 zero} yields
 $	\E\left[ 	\abs*{{\revarn \ve^{-1}\nabla_\theta Q_\ve(\theta^\star)^{(1)} }
 	- \dot{Q}^1_0}^p \right] \le C {\ve}^p$, 
 and in turn ${\revarn \ve^{-1}\nabla_\theta Q_\ve(\theta^\star)^{(1)} }$
 converges to $\dot{Q}^1_0$ in $\mathbf{L}^p$.
\end{proof}	

Let us introduce the following $d_\Theta \times d_\Theta$ matrix, defined by
\begin{multline}\label{Eq : def cal J}
	\mathcal{J}_{u,v}(\theta)=2 \int_0^T \left( \int_0^t K(t-s) 
	\frac{\partial}{{\revarn\partial}\theta_u} (b(X^0_s(\theta),\theta)) ds \right)^*
\\\times\left( \int_0^t K(t-s) 
\frac{\partial}{{\revarn\partial}\theta_v} (b(X^0_s(\theta),\theta)) ds \right) dt,
\end{multline}
where $(u,v) \in \{1,\dots,d_\Theta\}$. Remark that
\begin{equation}	
	\label{Eq : partial b theta explicit}
	\frac{\partial}{\theta_u} (b(X^0_s(\theta),\theta))= \nabla_x b(X^0_s(\theta),\theta) \times (Y^0_s(\theta))_{\cdot,u} +
\nabla_\theta b(X^0_s(\theta),\theta)_{\revarn \cdot,u} \,.
\end{equation}
In the sequel, we will need the following assumption :
\begin{assumption} \label{Ass : positive matrix}
	We have $\det\left( \mathcal{J}(\theta^\star) \right) >0$. 
\end{assumption}
Before stating the central limit theorem for the estimator $\hat{\theta}_\ve$, we present a simple proposition, concerning the validity of \ref{Ass : positive matrix}. Its proof is postponed to the Appendix.
\begin{proposition} \label{prop : grad b invert A3}
{\revarn	Assume \ref{Ass : global Lip}, {\revK \ref{Ass : kernel}}, and that $b\in\mathcal{C}^1_P(\mathbb{R}^d\times\overset{\circ}{\Theta},\mathbb{R}^d)$. 
	The condition
	$\det \left(\nabla_\theta b(x_0,\theta^\star)^*\nabla_\theta b(x_0,\theta^\star)\right)\neq 0$ implies \ref{Ass : positive matrix}.}
\end{proposition}
\begin{theorem} \label{theo : CLT}
	Assume \ref{Ass : global Lip}, \ref{Ass : identif}, {\revK \ref{Ass : kernel}}, \ref{Ass : positive matrix}, that
	 $b \in \mathcal{C}^{2}_P(\mathbb{R}^d\times\overset{\circ}{\Theta},\mathbb{R}^d)$ and $\theta^\star \in \overset{\circ}{\Theta}$. Then,
	\begin{equation} \label{Eq : asymptotic est proba}
		\frac{1}{\ve}(\hat{\theta}_\ve-\theta^\star) \xrightarrow[\mathbb{P}]{n\to\infty}
		- \mathcal{J}(\theta^\star)^{-1} \times \big(\dot{Q}_0^{1}+\dot{Q}_0^{2} \big)^*.
	\end{equation}
	If we furthermore assume  {\revarn \ref{Ass : identif_strong}[$\rho$] for some $\rho>0$,} then the convergence holds in $\mathbf{L}^p$ for all $p\ge 1$.	
\end{theorem}
\begin{proof}
	Under the assumption that $b \in \mathcal{C}^{2}_P$, {\revarn the mapping $\theta \mapsto X^0_t(\theta)$ is of class $\mathcal{C}^2$ and its derivatives are given as solutions} of 
		the formal differentiation of the Volterra equation \eqref{Eq : Volterra no noise} (see Remark \ref{R : rem deriv} in the Appendix for a justification.).
	This allows {\revarn us} to differentiate twice the contrast function $Q_\ve$. An expression for the first derivative is given in \eqref{Eq : Nabla Q_ve split}, 
	which we rewrite
	\begin{multline*}
		\frac{\partial Q_\ve}{\partial \theta_u}(\theta) = 	
		-2
		\int_0^T \Big( \int_0^t K(t-s)\left(b(X^0_s(\theta^\star),\theta^\star) - b(X^0_s(\theta),\theta)\right)ds 
		\\ + \ve \int_0^t K(t-s)a(X^\ve_s) dB_s 	\Big)^* 
		\\ \times 
		\Big( \int_0^t K(t-s)\frac{\partial}{\partial \theta_u}\left(b(X^0_s(\theta),\theta)\right) ds \Big) dt.
	\end{multline*}
	Differentiating once more, we get
	\begin{multline*}
		\frac{\partial Q_\ve}{\partial \theta_u \partial \theta_v} (\theta)= 	
		2
		\int_0^T \left( \int_0^t K(t-s)\frac{\partial}{\partial \theta_u}\left(b(X^0_s(\theta),\theta)\right) ds  \right)^* 
		\\ \times 
		\left( \int_0^t K(t-s)\frac{\partial}{\partial \theta_{\revarn v}}\left(b(X^0_s(\theta),\theta)\right) ds \right) dt
		+\mathcal{E}_{u,v}(\theta)
	\end{multline*}
	where 
	\begin{multline}\label{Eq : def mathcal E in CLT}
		\mathcal{E}_{u,v}(\theta)=
		-2
		\int_0^T \Big( \int_0^t K(t-s)\left(b(X^0_s(\theta^\star),\theta^\star) - b(X^0_s(\theta),\theta)\right)ds 
	\\	+ \ve \int_0^t K(t-s)a(X^\ve_s) dB_s 	\Big)^* 
		\\ \times 
		\Big( \int_0^t K(t-s)\frac{\partial^2}{\partial \theta_u\partial \theta_v}\left(b(X^0_s(\theta),\theta)\right) ds \Big) dt.
	\end{multline}
	This expression gives $	\frac{\partial Q_\ve}{\partial \theta_u \partial \theta_v} (\theta)=\mathcal{J}_{u,v}(\theta)+
	\mathcal{E}_{u,v}(\theta)$.
		
Now, using that $\hat{\theta}_\ve$ is a minimizer of the contrast function, we write for $\hat{\theta}_\ve \in \overset{\circ}{\Theta}$,  $\nabla_\theta Q_\ve(\hat{\theta}_\ve)^* = 0 = 
	\nabla_\theta Q_\ve({\theta}^\star)^* + \int_0^1 D^2_{\theta\theta} Q_\ve(\tilde{\theta}_\ve(r))(\hat{\theta}_\ve-\theta^\star)dr$,
	where $D^2_{\theta\theta} Q_\ve(\theta)$ is the matrix $\begin{bmatrix} \frac{\partial^2 Q_\ve(\theta)}{\partial \theta_u \partial \theta_v} \end{bmatrix}_{u,v}$ and
	$\tilde{\theta}_\ve(r)=\theta^\star+r(\hat{\theta}_\ve-\theta^\star)$.
	On the event $  \{\det \big(\int_0^1 D^2_{\theta\theta} Q_\ve(\tilde{\theta}_\ve(s)) ds\big) \neq 0;~ \hat{\theta}_\ve \in \overset{\circ}{\Theta} \}$, we have
	\begin{equation} \label{Eq : fond CLT est}
		\ve^{-1}(\hat{\theta}_\ve - \theta^\star)=- \left(\int_0^1 D^2_{\theta\theta} Q_\ve(\tilde{\theta}_\ve(r))dr \right)^{-1} \times \frac{ 	\nabla_\theta Q_\ve({\theta}^\star)^*}{\ve}.
	\end{equation}
We now prove for all $(u,v)\in \{1,\dots,d_\Theta\}$,
\begin{equation}\label{Eq : cv proba cal J}
		\int_0^1 \frac{\partial^2 Q_\ve}{\partial \theta_u\partial \theta_v}(\tilde{\theta}_\ve(r))dr \xrightarrow[\mathbb{P}]{\ve \to 0}  \mathcal{J}_{u,v}(\theta^\star) .
\end{equation}
From Proposition \ref{prop :  consistency}, we know that $\hat{\theta}_\ve \xrightarrow{\mathbb{P}} \theta^\star$ and in consequence 
$\sup_{r \in [0,1]}\abs{\tilde{\theta}_\ve(r) -\theta^\star} \xrightarrow{\mathbb{P}}0$. The function $\theta \mapsto \mathcal{J}(\theta)$ is $\mathcal{C}^1$, using the expression
\eqref{Eq : def cal J}--\eqref{Eq : partial b theta explicit} and the fact that $\theta \mapsto X^0_s(\theta)$ is of class $\mathcal{C}^2$. Thus, we deduce
that $\int_0^1\mathcal{J}_{u,v}(\tilde{\theta}_\ve(r))dr \xrightarrow{\mathbb{P}} \mathcal{J}_{u,v}({\theta}^\star)$.

Now, remark that  $(s,\theta) \mapsto \frac{\partial^2}{\partial \theta_u\partial \theta_v}\left(b(X^0_s(\theta),\theta)\right)$ is bounded on $[0,T]\times\Theta$. Also, we have 
$ \abs{\int_0^t K(t-s)\left(b(X^0_s(\theta^\star),\theta^\star) - b(X^0_s(\theta),\theta)\right)ds} \le c |\theta^\star-\theta|$, using that $\theta \mapsto b(X^0_s(\theta),\theta)$ is $\mathcal{C}^1$ with a bounded derivative. It implies 
\begin{multline}\label{Eq : Lip part cal E}
\sup_{r\in [0,1]}	\Bigg|
	\int_0^T \left( \int_0^t K(t-s)\left(b(X^0_s(\theta^\star),\theta^\star) - b(X^0_s(\tilde{\theta}_\ve(r)),\tilde{\theta}_\ve(r))\right)ds 
	\right)^* 
	\\ \times 
	\left( \int_0^t K(t-s)\frac{\partial^2}{\partial \theta_u\partial \theta_v}\left(b(X^0_s(\tilde{\theta}_\ve(r)),\tilde{\theta}_\ve(r))\right) ds \right) dt
	\Bigg| \\ \le c \sup_{r\in[0,1]}\abs{\tilde{\theta}_\ve(r) - \theta^\star} \le c\abs{\hat{\theta}_\ve - \theta^\star} 
	\xrightarrow[\mathbb{P}]{\ve\to0}0.
\end{multline}
Moreover, using \eqref{Eq : moment Volterra}, we have
\begin{multline*}
	\ve \sup_{r \in [0,1]} \Big|
\int_0^T \left(  \int_0^t K(t-s)a(X^\ve_s) dB_s 	\right)^*  \\\times
\left( \int_0^t K(t-s)\frac{\partial^2}{\partial \theta_u\partial \theta_v}\left(b(X^0_s(\tilde{\theta}_\ve(r)),\tilde{\theta}_\ve(r))\right) ds \right) dt \Big|
\\ \le c \ve \int_0^T  \abs*{\int_0^t K(t-s)a(X^\ve_s) dB_s} dt
\xrightarrow[\mathbb{P}]{\ve\to0} 0.
\end{multline*}
Hence $\int_0^1\abs{\mathcal{E}_{u,v}(\tilde{\theta}_\ve(r))}dr \xrightarrow[\mathbb{P}]{\ve \to 0}0$, and 
{\revarn it implies} \eqref{Eq : cv proba cal J}.

Now, using Assumption \ref{Ass : positive matrix}, we deduce that 
$\P\left( \det\big( \int_0^1 D_{\theta\theta}^2 Q_\ve(\tilde{\theta}_\ve(r))dr\big) \neq 0\right)\xrightarrow{\ve\to0}1$. Recall that on this set, we have $	\ve^{-1}(\hat{\theta}_\ve - \theta^\star)=- \bigl( \int_0^1 D^2_{\theta\theta} Q_\ve(\tilde{\theta}_\ve(r))dr\bigr)^{-1} \times \frac{ 	\nabla_\theta Q_\ve({\theta}^\star)^*}{\ve}$.
{\revarn Then, \eqref{Eq : cv proba cal J} and Proposition \ref{lem :  cv grad Q} yield,}
\begin{equation*}
	- \bigl( \int_0^1 D^2_{\theta\theta} Q_\ve(\tilde{\theta}_\ve(r))dr\bigr)^{-1} \times \frac{ 	\nabla_\theta Q_\ve({\theta}^\star)^*}{\ve} \xrightarrow[\mathbb{P}]{\ve\to 0}
		- \mathcal{J}(\theta^\star)\times\bigl(\dot{Q}_0^{1}+\dot{Q}_0^{2}\bigr).
\end{equation*}
This shows the convergence \eqref{Eq : asymptotic est proba}.

We now assume {\revarn \ref{Ass : identif_strong}[$\rho$] with $\rho>0$.} Let us prove that the convergence \eqref{Eq : asymptotic est proba} holds in $\mathbf{L}^p$. It is sufficient to show
\begin{equation} \label{Eq : suffi for cv Lp}
	\forall p \ge 1, \quad \sup_{0<\ve\le 1} \E\left[\abs[\big]{\ve^{-1}(\hat{\theta}_\ve - \theta^\star)}^p\right]<\infty.
\end{equation}
Using that $\mathcal{J}(\theta^\star)$ is {\revarn symmetric positive definite,} by Assumption \ref{Ass : positive matrix}, we have $\mathcal{J}(\theta^\star) > \eta \text{Id}$ for some $\eta >0$ and $\text{Id}$ is the identity matrix of size $d_\Theta$. Now let $A_\ve =\{ \norm{
	\int_0^1 D^2_{\theta\theta} Q_\ve(\tilde{\theta}_\ve(r))dr-\mathcal{J}(\theta^\star)}_{H.S.} \le \eta/2\}$ where $\norm{Q}_{H.S}=Tr(QQ^*)^{1/2}$ is the Hilbert-Schmidt norm of matrices. On $A_\ve$, we have $\int_0^1D^2_{\theta\theta} Q_\ve(\tilde{\theta}_\ve(r))dr \ge \frac{\eta}{2} \text{Id }$, and $\int_0^1D^2_{\theta\theta} Q_\ve(\tilde{\theta}_\ve(r))dr$ is invertible. Also, we denote $B_\ve=\{ \abs{\hat{\theta}_\ve-\theta^\star}<\eta'\}$ where  $\eta'$ is such that the ball with radius $\eta'$ and centred at $\theta^\star$ is included in $\overset{\circ}{\Theta}$.
The representation \eqref{Eq : fond CLT est} allows us to write
\begin{equation*}
	\ve^{-1}(\hat{\theta}_\ve-\theta^\star)=
	\ve^{-1}(\hat{\theta}_\ve-\theta^\star) \one_{\{A_\ve^c \cup B_\ve^c \}} - \Bigl(\int_0^1D^2_{\theta\theta} Q_\ve(\tilde{\theta}_\ve(r))dr\Bigr)^{-1} \times \frac{\nabla_\theta Q_\ve(\theta^\star)}{\ve}
	\one_{\{A_\ve\cap B_\ve\}}
\end{equation*}
Using that $\Theta$ is a compact set, and the boundedness of $\bigl(\int_0^1 D^2_{\theta\theta} Q_\ve(\tilde{\theta}_\ve(r))dr\bigr)^{-1}$ on $A_\ve$, it yields
\begin{equation*}
\abs[\big]{\ve^{-1}(\hat{\theta}_\ve-\theta^\star)}\le \frac{c}{\ve}  \one_{\{A_\ve^c\cup B_\ve^c\}} + c \abs*{\frac{\nabla_\theta Q_\ve(\theta^\star)}{\ve}}.
\end{equation*}
Now Lemma \ref{lem :  cv grad Q} implies $\E\big[\abs{\frac{\nabla_\theta Q_\ve(\theta^\star)}{\ve}}^p\big] \le c(p)$, and we deduce
\begin{equation}\label{Eq : proof TLC majo mom}
\E \left[\abs[\big]{\ve^{-1}(\hat{\theta}_\ve-\theta^\star)}^p\right]	\le c(p) \ve^{-p}  \P( A_\ve^c ) +c(p)\ve^{-p}  \P( B_\ve^c ) + c(p).
\end{equation}
It remains to upper bound $\P( A_\ve^c )$ and $\P( B_\ve^c )$. We write for any $q\ge 1$,
\begin{align*}
	\P( A_\ve^c ) &= \P\left( \norm[\big]{\int_0^1D^2_{\theta\theta} Q_\ve(\tilde{\theta}_\ve(r))dr-\mathcal{J}(\theta^\star)}_{H.S.} > \eta/2\right)
	\\
	&\le \E\left[  \norm[\big]{\int_0^1D^2_{\theta\theta} Q_\ve(\tilde{\theta}_\ve(r))dr-\mathcal{J}(\theta^\star)}_{H.S.}^q \right] (2/\eta)^q.
\end{align*}
As we have $D^2_{\theta\theta}Q_\ve(\theta)=\mathcal{J}(\theta)+\mathcal{E}(\theta)$, where the matrix $\mathcal{E}$ is defined by \eqref{Eq : def mathcal E in CLT}.
Using that $\theta \mapsto \mathcal{J}(\theta)$ is $\mathcal{C}^1$ with bounded derivatives, \eqref{Eq : Lip part cal E}, and 
$\sup_{s,\theta} \abs{\frac{\partial^2}{\partial \theta_u\partial \theta_v}\left(b(X^0_s({\theta}),{\theta})\right)} \le c$ we have
\begin{equation*}
	\norm[\big]{\int_0^1D^2_{\theta\theta} Q_\ve(\tilde{\theta}_\ve(r))dr-\mathcal{J}(\theta^\star)}_{H.S.}
		\le c \abs{\hat{\theta}_\ve 	-\theta^\star} + c\ve \int_0^T \abs{\int_0^t K(t-s)a(X^\ve_s)dB_s}dt.
\end{equation*}
From \eqref{Eq : moment Volterra}, we deduce
\begin{equation*}
\E\left[	\norm[\big]{\int_0^1 D^2_{\theta\theta} Q_\ve(\tilde{\theta}_\ve(r))dr-\mathcal{J}(\theta^\star)}_{H.S.}^q \right]
	\le c \E[\abs{\hat{\theta}_\ve-\theta^\star}^q] + c\ve^q.
\end{equation*}
{\revarn By Lemma \ref{lem : mino Q}, under Assumption \ref{Ass : identif_strong}[$\rho$], we know that 
\ref{Ass : identif_strong_contrast}[$\rho'$] holds with some $\rho'>0$. Consequently, we can use \eqref{Eq : consistency Lp}.} 
We deduce
\begin{equation*}
	\E\left[	\norm[\big]{\int_0^1 D^2_{\theta\theta} Q_\ve(\tilde{\theta}_\ve(r))dr-\mathcal{J}(\theta^\star)}_{H.S.}^q \right]
	\le c(q) [\ve^{\revarn  q/\rho'} + \ve^q],
\end{equation*}
{\revarn where $\rho'>0$} is fixed and $q\ge 1$ arbitrary. 
 It gives 
 \begin{equation}\label{Eq :  upper P Ave}
 	\P(A_\ve^c) \le c(q) [\ve^{\revarn  q/\rho'} + \ve^q] (2/\eta)^q.
 \end{equation}
 Now, $	\P(B_\ve^c)={\revarn \P}(\abs{\hat{\theta}_\ve-\theta^\star}\ge\eta') 
 \le \frac{{\revarn \E \big(}\abs{\hat{\theta}_\ve-\theta^\star}^q{\revarn \big)}}{\eta'^q}$ for any $q \ge 1$. Using \eqref{Eq : consistency Lp}, it gives $	\P(B_\ve^c) \le c(q) \frac{\ve^{\revarn  q/\rho'}}{\eta'^q}$. Collecting this with \eqref{Eq :  upper P Ave} and \eqref{Eq : proof TLC majo mom} entails
\begin{equation*}
	\E \left[\abs{\ve^{-1}(\hat{\theta}_\ve-\theta^\star)}^p\right]	\le c(p) \ve^{-p}  c(q) [\ve^{\revarn  q/\rho'} + \ve^q] (2/\eta)^q  + c(p) 
	\ve^{-p}  c(q) \ve^{\revarn  q/\rho'}  (1/\eta')^q
	+c(p).
\end{equation*}
By choosing $q$ with {\revarn $ q/\rho' \ge p$,} we deduce that the right-hand side in the equation above is bounded independently of $\ve\in (0,1]$. As a result \eqref{Eq : suffi for cv Lp} is proved.
\end{proof}
	\begin{remark}
	Whereas the limiting distribution of the estimator is given by the explicit {\revarn random} variable \eqref{Eq : asymptotic est proba}, its expression is cumbersome in most models. In particular, the dependence of the asymptotic variance of the estimator {\revarn on the parameters or on the time horizon} $T$
	is not obvious on this formula.  In the next section, we provide a numerical assessment of the asymptotic variance in a specific model. 
	
	{\revarn The reference \cite{kutoyantsIdentificationDynamicalSystems1994} establishes a CLT in a situation of a constant volatility coefficient and a non-singular kernel. It can be verified that the variance of the right-hand side of \eqref{Eq : asymptotic est proba} coincides with the one given in Theorem 7.5 of \cite{kutoyantsIdentificationDynamicalSystems1994}.}
\end{remark}
\section{Numerical simulations}\label{S : num sim}
To assess the {\revarn performance} of the estimation procedure in practice, we consider a Monte Carlo experiment with 500 replications and for different values of $\ve$. We use a Volterra Ornstein-Uhlenbeck process,  
\begin{equation*}
	X^\varepsilon_t=X^\varepsilon_0+\ve \int_0^t K(t-s) dB_s
	+ \int_0^t K(t-s) (\theta_0 X_s^\ve+\theta_1) ds,
\end{equation*}
{\revK with $K=K_0$.}
As it is impossible to continuously observe the path $(X^\ve_t)_{t\in [0,T]}$, the contrast function \eqref{Eq : contrast TFE} is replaced by a finite sum 
	$Q_{\ve,\Delta}(\theta)=\sum_{i=0}^{\lfloor T/\Delta \rfloor} \abs{X^\ve_{i\Delta}-X^0_{i\Delta}(\theta)}^2$ based on the observations of the sampling $(X^\ve_{i\Delta})_{i=0,\dots,\lfloor T/\Delta \rfloor }$. The solution of \eqref{Eq : Volterra no noise} is computed numerically using an Euler scheme with mesh $\delta=\Delta/50$ much finer than the observation sampling. The true parameter values  
	are  $(\theta_0,\theta_1)=(-1,1)$.
	
	{\revarn In the case where $\alpha=0.3$, Tables \ref{T : Num T 1} and \ref{T : Num T 10} report} the empirical means, and the empirical standard deviations rescaled by the factor $\ve^{-1}$. The values of $\ve$ range from $1/2$ to $1/100$ while $\Delta$ ranges from $1/5$ to $1/100$. The estimator {\revarn performs} well in this example as soon as $\epsilon \le 1/10$,  and can be {\revarn highly} biased if $\ve$ is not small enough.  On the other hand, the quality of estimation remains good even for large values of $\Delta$.  It is {\revarn noteworthy} that the estimator has a {\revarn smaller variance for $T=10$ than for $T=1$.} However, if we increase $T$ further, we do not have any additional improvement on the variance. For instance, if $T=50$, $\ve=1/100$, $\Delta=1/100$, the rescaled empirical standard deviation is
	$(2.6, 2.5)$. Consequently, {\revarn despite the} fact that we are estimating drift parameters, the asymptotic variance {\revarn does} not seem to 
	{\revarn converge} to $0$ as $T \to \infty$. 
	
	{\revarn In Tables \ref{T : Num T 1_alpha_small} and \ref{T : Num T 10_alpha_small}, we display the results for $\alpha=0.1$. The standard deviation of the estimator is clearly larger for this smaller value of $\alpha$ than for $\alpha=0.3$.}	

\section{Appendix}\label{S: Appendix}
	Let $\mathfrak{b} : [0,T]\times\mathbb{R}^q \times \Theta \to \mathbb{R}^q $ {\revarn be} such that for all $t\in[0,T]$, $(x,\theta)\mapsto \mathfrak{b}(t,x,\theta)$ is $\mathcal{C}^1(\mathbb{R}^q\times\overset{\circ}{\Theta})$ and the functions $\nabla_x \mathfrak{b}(t,x,\theta)$ and $\nabla_\theta \mathfrak{b}(t,x,\theta)$ are $\mathcal{C}^0([0,T]\times\mathbb{R}^q\times\overset{\circ}{
		\Theta})$. Moreover, we assume that for some $c>0$.
\begin{equation} \label{Eq : Appendix unif poly}
	\sup_{t \in [0,T]; \theta \in \overset{\circ}{\Theta}} \abs{\mathfrak{b}(t,x,\theta)}+ \abs{\nabla_x \mathfrak{b}(t,x,\theta)} + \abs{\nabla_\theta\mathfrak{b}(t,x,\theta)} \le c (1+ \abs{x}^c) .
\end{equation}
We define $(\mathcal{X}_t(\theta))_t$ and $(\mathcal{Y}_t(\theta))_t$ {\revarn as the solutions to} the following two deterministic Volterra equations 
\begin{align}\label{Eq : Appendix Volterra}
	&\mathcal{X}_t(\theta)=x_0 + \int_0^t K(t-s) \mathfrak{b}(s,\mathcal{X}_s(\theta),\theta) ds, \quad x_0 \in \mathbb{R}^q {\revarn,} \\
\label{Eq : Appendix Volterra derivative}
	&\mathcal{Y}_t(\theta)=\int_0^t K(t-s)\big[\nabla_x \mathfrak{b}(s,\mathcal{X}_s(\theta),\theta) \mathcal{Y}_s(\theta) + 
	\nabla_\theta \mathfrak{b} (s,\mathcal{X}_s(\theta),\theta)\big]ds{\revarn .}
\end{align}
We know from Theorem 1.1. in \cite{wangExistenceUniquenessSolutions2008} that the solutions of these equations exist. 
\begin{lemma} \label{lem : X C1}
The function $\theta\mapsto\mathcal{X}(\theta)$ is $\mathcal{C}^1(\overset{\circ}{\Theta},\mathbb{R}^q)$, and we have $\nabla_\theta \mathcal{X}_t(\theta)=\mathcal{Y}_t(\theta)$.
\end{lemma}
\begin{remark}\label{R : rem deriv}
	If $b \in \mathcal{C}^{1}_P(\mathbb{R}^d\times\overset{\circ}{\Theta},\mathbb{R}^d)$, we can apply the lemma with $\mathfrak{b}(t,x,\theta)=b(x,\theta)$ and deduce that $\theta \mapsto X^0(\theta)$ is $\mathcal{C}^1$ with a derivative solution of 
	\eqref{Eq : ODE Volterra Y}. If $b \in \mathcal{C}^{2}_P$, we can apply the lemma with $\mathfrak{b}(t,y,\theta)=\nabla_xb(X^0_t(\theta),\theta)y+
	\nabla_\theta b(X^0_t(\theta),\theta)$ and get that  $\theta \mapsto X^0_t(\theta)$ is a function of class $\mathcal{C}^2$.
\end{remark}
\begin{proof}[Proof of Lemma \ref{lem : X C1}]
	For $\theta \in \overset{\circ}{\Theta}$ and $h\in \mathbb{R}^{d_\Theta}$ with $\theta+h \in \Theta$, we define 
	$\overline{\mathcal{Y}}^h_t(\theta)= \mathcal{X}_t(\theta+h)-\mathcal{X}_t(\theta)-\mathcal{Y}_t(\theta)h$, which satisfies
	\begin{multline*}
		\overline{\mathcal{Y}}^h_t(\theta)=
		\int_0^t K(t-s) \big[  
		\mathfrak{b}(s,\mathcal{X}_s(\theta+h),\theta+h)-\mathfrak{b}(s,\mathcal{X}_s(\theta),\theta)-\\
		\nabla_x \mathfrak{b}(s,\mathcal{X}_s(\theta),\theta) \mathcal{Y}_s(\theta)h -
		\nabla_\theta \mathfrak{b} (s,\mathcal{X}_s(\theta),\theta)h\big]ds.
	\end{multline*}
By {\revarn the} fundamental theorem of calculus applied to the function $\xi\in[0,1] \mapsto \mathfrak{b} \big(s,\mathcal{X}_s(\theta) + \xi 
(\mathcal{X}_s(\theta+h)-\mathcal{X}_s(\theta)) ,\theta +\xi h \big)$, {\revarn we have }
\begin{multline*}
	\mathfrak{b} \big(s, \mathcal{X}_s(\theta+h),\theta+h \big)-\mathfrak{b} \big(s,\mathcal{X}_s(\theta),\theta \big)
=\\ \int_0^1 \big[
	\nabla_x \mathfrak{b}(s,\mathcal{X}^{\xi}_s(\theta),\theta+ \xi h) \big( \mathcal{X}_s(\theta+h)-\mathcal{X}_s(\theta) \big) 
	+
\nabla_\theta \mathfrak{b} (s,\mathcal{X}^{\xi}_s,\theta+ \xi h)h\big]d\xi ,
\end{multline*}
where $\mathcal{X}^{\xi}_s=\mathcal{X}_s(\theta) + \xi 
(\mathcal{X}_s(\theta+h)-\mathcal{X}_s(\theta))  $.
We deduce that
\begin{equation}\label{Eq : Appendix Gronwall barY}
	\overline{\mathcal{Y}}^h_t(\theta)=
\int_0^t K(t-s)  \int_0^1 \nabla_x \mathfrak{b}(s,\mathcal{X}^{\xi}_s(\theta),\theta+\xi h) \overline{\mathcal{Y}}^h_s(\theta) d\xi ds +  h \times \mathcal{R}_t(\theta,h)
\end{equation}
	where
	\begin{multline*}
		\mathcal{R}_t(\theta,	h)=
		\int_0^t K(t-s)\int_0^1\big[ \nabla_x \mathfrak{b}(s,\mathcal{X}^{\xi}_s(\theta),\theta+\xi h) - \nabla_x \mathfrak{b}(s,\mathcal{X}_s(\theta),\theta) \big]d\xi \mathcal{Y}_s(\theta)ds +
		\\
		\int_0^t K(t-s)\int_0^1\big[ \nabla_\theta \mathfrak{b}(s,\mathcal{X}^{\xi}_s(\theta),\theta+\xi h) - \nabla_\theta \mathfrak{b}(s,\mathcal{X}_s(\theta),\theta) \big] d\xi ds .
		\end{multline*}
From Lemma 2.2 in \cite{wangExistenceUniquenessSolutions2008}, we know that $\sup_{t\in[0,T]} \abs{\mathcal{X}_t(\theta) }\le c$, and the inspection of the proof shows that constant $c$ depends on the coefficients of the differential equation defining $\mathcal{X}_t(\theta)$, and is independent of $\theta$ under the condition \eqref{Eq : Appendix unif poly}. Analogously,  $\sup_{s\in[0,T],\theta\in \overset{\circ}{\Theta}} \abs{\mathcal{Y}_t(\theta)} \le c$.
Let us denote
\begin{multline*}
\epsilon(\theta,h) =\sup_{s \in [0,T], \xi \in [0,1]}
\abs{\nabla_x \mathfrak{b}(s,\mathcal{X}_s(\theta) + \xi (\mathcal{X}_s(\theta+h)-X_s(\theta) ),\theta+\xi h) - \nabla_x \mathfrak{b}(s,\mathcal{X}_s(\theta),\theta) }
\\+ \abs{\nabla_\theta \mathfrak{b}(s,\mathcal{X}_s(\theta) + \xi (\mathcal{X}_s(\theta+h)-\mathcal{X}_s(\theta) ),\theta+\xi h) - \nabla_\theta \mathfrak{b}(s,X_s(\theta),\theta) }.
\end{multline*}
Using that $K\in \mathbf{L}^1([0,T])$ and the boundedness of $\mathcal{Y}_s(\theta)$, we have $\abs{\mathcal{R}_t(h)} \le  c \epsilon(\theta,h)$ for all $t \in [0,T]$. The equation \eqref{Eq : Appendix Gronwall barY} implies, using \eqref{Eq : Appendix unif poly} for $\nabla_x \mathfrak{b}$,
\begin{equation*}
	\abs*{\frac{\overline{\mathcal{Y}}^h_t(\theta)}{h}}\le
	c \int_0^t {\revK\abs*{K(t-s)}}	\abs*{\frac{\overline{\mathcal{Y}}^h_s(\theta)}{h}} ds + c \epsilon(\theta,h).
\end{equation*}
We use {\revK \eqref{Eq : comp K K0} and Gronwall's lemma} (Corollary 2 in \cite{yeGeneralizedGronwallInequality2007}) to deduce that for some $c >0$ and all $t\in [0,T]$,
\begin{equation*}
\abs*{\frac{\overline{\mathcal{Y}}^h_t(\theta)}{h}} \le c \epsilon(\theta,h).
\end{equation*}
Now, from the definition of $\overline{\mathcal{Y}}^h_t(\theta)$, the fact that $\nabla_\theta \mathcal{X}_t(\theta)= \mathcal{Y}_t(\theta)$ 
{\revarn will} be proved if we show that 
$\epsilon(\theta,h) \xrightarrow{h\to0}0$. Since $\nabla_b \mathfrak{b}$ and $\nabla_\theta \mathfrak{b}$ are continuous it is sufficient to show
\begin{equation}\label{EQ : X ctn in theta}
	\sup_{s\in [0,T]} \abs*{ \mathcal{X}_t(\theta+h)- \mathcal{X}_t(\theta)} \xrightarrow{h\to0}0.
\end{equation}
However, using the boundedness on {\revarn a} compact set of $\nabla_b \mathfrak{b}$ and $\nabla_x \mathfrak{b}$, it is possible to get from
\eqref{Eq : Appendix Volterra} that
\begin{multline*}
	\abs{\mathcal{X}_t(\theta+h)-\mathcal{X}_t(\theta)}\le 
	c\int_0^t {\revK\abs*{K(t-s)}}\abs{\mathcal{X}_s(\theta+h)-\mathcal{X}_t(\theta)} ds + c\int_0^t {\revK\abs*{K(t-s)}} h ds
	\\
\le	c\int_0^t {\revK\abs*{K(t-s)}}\abs{\mathcal{X}_s(\theta+h)-\mathcal{X}_t(\theta)} ds + c h .
\end{multline*}
The Corollary 2 in \cite{yeGeneralizedGronwallInequality2007} {\revK with \eqref{Eq : comp K K0}} yields $\sup_{t\in[0,T]} 	\abs{\mathcal{X}_t(\theta+h)-\mathcal{X}_t(\theta)} \le c h$.
 Thus \eqref{EQ : X ctn in theta} is proved and in turn we deduce  $\nabla_\theta \mathcal{X}_t(\theta)= \mathcal{Y}_t(\theta)$.

To finish the proof of the lemma, we have to show the continuity of $\theta \mapsto  \mathcal{Y}_t(\theta)$. We omit the details on the proof of this continuity, which relies on the continuity of the coefficients of the Volterra equation  \eqref{Eq : Appendix Volterra derivative}, and on an application of the Volterra Gronwall's inequality.
\end{proof}
{\revarn 
We conclude this appendix by a proof of the Proposition \ref{prop : grad b invert A3}.
\begin{proof}[Proof of Proposition \ref{prop : grad b invert A3}]
	By contradiction, assume that \ref{Ass : positive matrix} does not hold. Since the symmetric matrix $\mathcal{J}(\theta^*)$ is positive semidefinite, 
	this implies that 
	$\det \left(\mathcal{J}(\theta^\star)\right)=0$.
	Hence, there exists $h\in\mathbb{R}^{d_\Theta} \setminus \{0\}$ such that $h^* \mathcal{J}(\theta^\star) h=0$. By \eqref{Eq : def cal J}, it follows  that
	\begin{equation*}
		0=\int_0^T \abs*{ \sum_{u=1}^{d_\Theta} h_u \int_0^t  K(t-s) 
			\frac{\partial}{{\revarn\partial}\theta_u} (b(X^0_s(\theta^\star),\theta^\star))
			ds }^2 dt.
	\end{equation*}
	If we define the $\mathbb{R}^d$ valued function $g$, by $g(s)=\sum_{u=1}^{d_\Theta} h_u 	\frac{\partial}{{\revarn\partial}\theta_u} (b(X^0_s(\theta^\star),\theta^\star))$, then we deduce that
	for all $t\in[0,T]$, $\int_0^t K(t-s)g(s)ds=0$. It is $K\star g(t)=0$ for all $t\in[0,T]$. By convolving with the kernel $L$, exactly as in the proof of Lemma \ref{lem : mino Q} 1), we deduce that $g(t)=0$ for all $t\in[0,T]$. Using the expression \eqref{Eq : partial b theta explicit} together with the fact that $Y^0_0(\theta^\star)=0$, we obtain
	\begin{equation*}
		0=g(0)=\sum_{u=1}^{d_\Theta} h_u \nabla_\theta b(x_0,\theta^\star)_{\cdot,u}
		= \nabla_\theta b(x_0,\theta^\star) h.
	\end{equation*}
Taking the square, this contradicts the assumption that $\nabla_\theta b(x_0,\theta^\star)^*\nabla_\theta b(x_0,\theta^\star)$ is an invertible matrix.	
\end{proof}
}

%
%
\printbibliography

\begin{table}[p] \caption{$\alpha=0.3$, $(\theta_0,\theta_1)=(-1,1)$, $T=1$} \label{T : Num T 1}
	\centering
	\begin{tblr}{
			cell{2}{1} = {r=2}{},
			cell{4}{1} = {r=2}{},
			cell{6}{1} = {r=2}{},
			cell{8}{1} = {r=2}{},
			vline{1-7} = {1}{},
			vline{1-7} = {2,4,6,8}{},
			vline{1-7} = {3,5,7,9}{},
			hline{1,2,4,6,8,10} = {-}{},
		}
		\backslashbox{$\Delta$}{$\ve$}	&  &1/2 & 1/10 & 1/20 & 1/100 \\
		1/5   & mean      & (-2,39, 2.01)    &  (-1.11, 1.06)   &  (-1.01, 1.01)    &   (-1.00, 1.00)    \\
		&  resc. std.   &  (10, 4.8)   & (8.8, 4.0)    &   (7.3, 3.3)   &  (7.2, 3.3)    \\
		1/10  &  mean     &   (-2.32, 1.81)  &  (-1.08, 1.05)   &  (-1.01, 1.01)    &  
		(-1.00, 1.00)   \\
		&  resc. std.   &   (10, 4.8)  &   (6.6, 3.3)   &  (6.4, 3.2)   &   (6.2, 3.0)   \\
		1/20  &mean       &   (-2.29, 1.73)  &  (-1.06, 1.04)   & (-1.02, 1.01)     &      (-1.00, 1.00)\\
		&  resc. std.   &   (8.0, 3.9)  &   (6.5, 3.2)  &   (6.2, 3.1)   & (6.1, 3.0)     \\
		1/100 & mean      &  (-2.17, 1.75)   &  (-1.07, 1.04)   &  (-1.02, 1.01)     &    (-1.00, 1.00)  \\
		& resc. std.    &  (7.9, 3.9)   & (6.5, 3.2)     &  (6.0, 3.0)   &   (5.9, 3.0)
	\end{tblr}
\end{table}

\begin{table}[p]\caption{$\alpha=0.3$, $(\theta_0,\theta_1)=(-1,1)$, $T=10$} \label{T : Num T 10}
	\centering
	\begin{tblr}{
			cell{2}{1} = {r=2}{},
			cell{4}{1} = {r=2}{},
			cell{6}{1} = {r=2}{},
			cell{8}{1} = {r=2}{},
			vline{1-7} = {1}{},
			vline{1-7} = {2,4,6,8}{},
			vline{1-7} = {3,5,7,9}{},
			hline{1,2,4,6,8,10} = {-}{},	
		}
		\backslashbox{$\Delta$}{$\ve$}	&  &1/2 & 1/10 & 1/20 & 1/100 \\
		1/5   & mean      &  (-1.92, 1.80)    &  (-1.03, 1.03)   &  (-1.01, 1.01)    &  (-1.00, 1.00)     \\
		&  resc. std.   &  (5.1, 4.4)    &  (2.4, 2.1)   &  (2.3, 2.0)    &  (2.3, 2.0)    \\
		1/10  &  mean     &  (-1.84, 1.72)    & (-1.03, 1.03)    & (-1.01, 1.01)     & 
		(-1.00, 1.00)     \\
		&  resc. std.   &   (4.5, 3.7)  & (2.4, 2.1)    & (2.3, 2.0)     &  (2.3, 2.0)    \\
		1/20  &mean       &   (-1.80, 1.68)  & (-1.02, 1.02)    &  (-1.01, 1.01)    &  (-1.00, 1.00)    \\
		&  resc. std.   &  (4.4, 3.6)   &   (2.3, 2.0)    &  (2.3, 2.0)    &   (2.3, 2.0)  \\
		1/100 & mean      &  (-1.71, 1.61)   &  (-1.03, 1.02)   &  (-1.01, 1.01)    & (-1.00, 1.00)     \\
		& resc. std.    &  (4.4, 3.6)   &  (2.3, 2.1)   &  (2.3, 2.0)    &    (2.3, 2.0)  
	\end{tblr}
\end{table}

\begin{table}[p] \caption{$\alpha=0.1$, $(\theta_0,\theta_1)=(-1,1)$, $T=1$} \label{T : Num T 1_alpha_small}
	\centering
	\begin{tblr}{
			cell{2}{1} = {r=2}{},
			cell{4}{1} = {r=2}{},
			cell{6}{1} = {r=2}{},
			cell{8}{1} = {r=2}{},
			vline{1-7} = {1}{},
			vline{1-7} = {2,4,6,8}{},
			vline{1-7} = {3,5,7,9}{},
			hline{1,2,4,6,8,10} = {-}{},
		}
		\backslashbox{$\Delta$}{$\ve$}	&  &1/2 & 1/10 & 1/20 & 1/100 \\
		1/5   & mean      & (-2,86, 2.59)    &  (-2.15, 1.76)   &  (-1.25, 1.17)    &   (-1.01, 1.00)    \\
		&  resc. std.   &  (9.3, 6.7)   & (28, 18)    &   (22, 14)   &  (11, 6.6)    \\
		1/10  &  mean     &   (-3.17, 2.60)  &  (-1.49, 1.32)   &  (-1.08, 1.06)    &  		(-1.00, 1.00)   \\
		&  resc. std.   &   (9.2, 6.6)  &   (18, 11)   &  (10, 6.4)   &   (8.2, 5.2)   \\
		1/20  &mean       &   (-2.98, 2.41)  &  (-1.30, 1.20)   & (-1.05, 1.04)     &      (-1.00, 1.00)\\
		&  resc. std.   &   (8.7, 6.0)  &   (12, 7.9)  &   (8.0, 5.1)   & (7.5, 4.9)     \\
		1/100 & mean      &  (-2.84, 2.25)   &  (-1.21, 1.14)   &  (-1.05, 1.04)     &    (-1.00, 1.00)  \\
		& resc. std.    &  (8.2, 5.6)   & (9.9, 6.1)     &  (7.5, 4.9)   &   (7.1, 4.6)
	\end{tblr}
\end{table}

\begin{table}[p]\caption{$\alpha=0.1$, $(\theta_0,\theta_1)=(-1,1)$, $T=10$} \label{T : Num T 10_alpha_small}
	\centering
	\begin{tblr}{
			cell{2}{1} = {r=2}{},
			cell{4}{1} = {r=2}{},
			cell{6}{1} = {r=2}{},
			cell{8}{1} = {r=2}{},
			vline{1-7} = {1}{},
			vline{1-7} = {2,4,6,8}{},
			vline{1-7} = {3,5,7,9}{},
			hline{1,2,4,6,8,10} = {-}{},	
		}
		\backslashbox{$\Delta$}{$\ve$}	&  &1/2 & 1/10 & 1/20 & 1/100 \\
		1/5   & mean      &  (-3.19, 2.93)    &  (-1.14, 1.12)   &  (-1.03, 1.03)    &  (-1.00, 1.00)     \\
		&  resc. std.   &  (7.6, 6.8)    &  (7.4, 6.6)   &  (4.2, 3.7)    &  (3.8, 3.4)    \\
		1/10  &  mean     &  (-3.07, 2.80)    & (-1.09, 1.08)    & (-1.03, 1.02)     & 	(-1.00, 1.00)     \\
		&  resc. std.   &   (7.6, 6.7)  & (4.6, 4.1)    & (3.7, 3.3)     &  (3.5, 3.1)    \\
		1/20  &mean       &   (-2.90, 2.70)  & (-1.08, 1.07)    &  (-1.02, 1.02)    &  (-1.00, 1.00)    \\
		&  resc. std.   &  (7.3, 6.5)   &   (4.1, 3.6)    &  (3.5, 3.1)    &   (3.4, 3.0)  \\
		1/100 & mean      &  (-2.62, 2.41)   &  (-1.08, 1.06)   &  (-1.02, 1.02)    & (-1.00, 1.00)     \\
		& resc. std.    &  (6.7, 5.9)   &  (4.0, 3.6)   &  (3.6, 3.2)    &    (3.4, 3.0)  
	\end{tblr}
\end{table}

\end{document}